  \documentclass[10pt,leqno]{article}

\usepackage[lite]{amsrefs}
\usepackage{latexsym,enumerate}

  \usepackage[notref, notcite]{}
\usepackage{amssymb, amsfonts, eucal,amsmath,amsthm}
\newtheorem{theorem}{Theorem}[section]

\newtheorem{lemma}[theorem]{Lemma}
\newtheorem{proposition}[theorem]{Proposition}
\newtheorem{corollary}[theorem]{Corollary}
\newtheorem{remark}[theorem]{Remark}

\newcommand{\sect}[1]{\section{#1} \setcounter{equation}{0} }

\newcommand{\norm}[2]{\left\|#1\right\|_{#2}}

\newcommand{\ds}{\displaystyle}
\newcommand{\Poly}{\mathbb P}
\newcommand{\Pn}{\mathbb P_n}
 \newcommand{\I}{{\mathcal I}}

\newcommand{\andd}{\quad\mbox{\rm and}\quad}
\newcommand\e{{\varepsilon}}

\def\be  {\begin{equation}}
\def\ee  {\end{equation}}

\newcommand{\B}{\mathbb S}
\newcommand{\C}{\mathbb C}
\newcommand{\W}{\mathbb L}
\newcommand{\R}{\mathbb R}
\newcommand{\N}{\mathbb N}

\newcommand{\ineq}[1]{\text{\rm(\ref{#1})}}
\newcommand{\ie}{{\em i.e., }}
\newcommand{\eg}{{\em e.g. }}

\newcommand{\st}{\;\; \big| \;\;}
\newcommand{\LL}{\mathbb{L}}
\newcommand{\Lp}{\LL_p}
\newcommand{\Lq}{\LL_q}
\newcommand{\tI}{\widetilde I}
\newcommand{\ttI}{\widehat{I}}

\newcommand{\wab}{w_{\alpha,\beta}}
\newcommand{\wba}{w_{\beta,\alpha}}
\newcommand{\wbb}{w_{\beta,\beta}}
\newcommand{\wb}{w_{0,\beta}}
\newcommand{\wg}{w_{\alpha -\gamma, \beta-\gamma}}
\newcommand{\AC}{\mathrm{AC}}
\newcommand{\loc}{\mathrm{loc}}
\newcommand\M {{\mathcal M}}
\newcommand\E{{\mathcal E}}

\newcommand{\thm}[1]{Theorem~\ref{#1}}
\newcommand{\lem}[1]{Lemma~\ref{#1}}
\newcommand{\cor}[1]{Corollary~\ref{#1}}
\newcommand{\prop}[2]{Proposition~\ref{#1}(\ref{#2})}

\newcommand{\rem}[1]{Remark~\ref{#1}}

 \newcommand{\Wpab}{\W_p^{\alpha,\beta}}
  \newcommand{\Wpabr}{\W_{p,r}^{\alpha,\beta}}
 \newcommand{\Wqab}{\W_q^{\alpha,\beta}}
 \newcommand{\Bpab}{\B_p^{\alpha,\beta}}
 \newcommand{\Bpba}{\B_p^{\beta,\alpha}}
  \newcommand{\Bpbb}{\B_p^{\beta,\beta}}
 \newcommand{\Wpbb}{\W_p^{\beta,\beta}}

\title{Weighted moduli of smoothness of $k$-monotone functions and applications }

\author{
Kirill A.  Kopotun\thanks{Department of Mathematics, University of
Manitoba, Winnipeg, Manitoba, R3T 2N2, Canada ({\tt
kopotunk@cc.umanitoba.ca}). Supported by NSERC of Canada.}    }

\begin{document}

 \maketitle

\begin{abstract}

Let $\omega_\varphi^k(f,\delta)_{w,\Lq}$ be the Ditzian-Totik modulus with weight $w$, $\M^k$ be the cone of  $k$-monotone functions on $(-1,1)$, \ie those functions
whose $k$th divided differences are nonnegative for all selections of $k+1$ distinct points  in $(-1,1)$, and denote
$\E (X, \Pn)_{w,q} := \sup_{f\in X} \inf_{P\in\Pn}\norm{w(f-P)}{\Lq}$,
where $\Pn$ is the set of algebraic polynomials of degree at most $n$.
Additionally, let
$\wab(x) := (1+x)^\alpha  (1-x)^\beta$ be the classical Jacobi weight,
and denote by $\Bpab$   the class  of all functions such that
$\norm{\wab   f}{\Lp}=1$.

In this paper, we determine the exact behavior (in terms of $\delta$) of $\sup_{f\in\Bpab \cap \M^k} \omega_\varphi^k(f,\delta)_{\wab,\Lq}$ for $1\leq p, q\leq \infty$ (the interesting case being $q<p$ as expected)
and $\alpha,\beta >-1/p$ (if $p<\infty$) or $\alpha,\beta\geq 0$ (if $p=\infty$). It is interesting to note that, in one case, the behavior is different for $\alpha=\beta=0$ and for $(\alpha,\beta)\neq (0,0)$.
Several applications  are given. For example, we determine the exact (in some sense) behavior of $\E (\M^k\cap \Bpab, \Pn)_{\wab,\Lq}$ for $\alpha,\beta \geq 0$.

\end{abstract}

%

\sect{Introduction and main results}

Let $\wab(x) := (1+x)^\alpha  (1-x)^\beta$ be the (classical) Jacobi weight, $\| \cdot \|_p := \norm{\cdot}{\Lp[-1,1]}$,
\[
\Wpab  := \left\{ f: [-1,1]\mapsto\R \st \norm{\wab   f}{p} < \infty \right\} ,
\]
and let $\Bpab$ be the unit sphere in $\Wpab$, \ie $f\in \Bpab$ iff $\norm{\wab   f}{p}=1$.
 It is   convenient to denote $J_p := (-1/p, \infty)$ if $p<\infty$, and $J_\infty := [0,\infty)$. Clearly,  $1\in\Wpab$ iff $\alpha,\beta \in J_p$.
 We note that  more general  than Jacobi weights can be considered, and many results in this paper are valid and/or can be modified to be valid for those general weights. However,   we only consider  Jacobi weights in order not to overcomplicate the proofs which are already rather technical, and since the estimates of rates of unweighted polynomial approximation that have matching converse results involve weighted moduli with classical Jacobi weights
 $w_{r/2, r/2} = \varphi^r$, $r\in\N$ (see \cites{kls, kls-ca} or   \ineq{direct} with $\alpha=\beta=0$ for an example of such an estimate). Here, as usual,   $\varphi(x):= w_{1/2,1/2} =(1-x^2)^{1/2}$.

Let
\[\Delta_h^k(f,x, [a,b]):=\left\{
\begin{array}{ll} \ds
\sum_{i=0}^k  {k \choose i}
(-1)^{k-i} f(x-kh/2+ih),&\mbox{\rm if }\, x\pm kh/2  \in [a,b] \,,\\
0,&\mbox{\rm otherwise},
\end{array}\right.\]
be the $k$th symmetric difference, $\Delta_h^k(f,x) := \Delta_h^k(f,x, [-1,1])$,  and let
\[
  \overrightarrow\Delta_h^k (f,x):= \Delta_h^k(f,x+kh/2) \andd
\overleftarrow\Delta_h^k (f,x):= \Delta_h^k(f,x-kh/2)
\]
be the forward and backward $k$th differences, respectively.
The weighted main part moduli and the weighted Ditzian-Totik (DT) moduli of smoothness (see \cite[(8.1.2), (8.2.10) and Appendix B]{dt}) are defined, respectively,  as
\[
 \Omega_\varphi^k(f,\delta)_{w,p} := \sup_{0<h\le \delta}\|w\Delta_{h\varphi}^k(f)\|_{\Lp[-1+2k^2h^2 ,1-2k^2h^2]}
\]
and
\be  \label{mod}
  \omega_\varphi^k(f,\delta)_{w,p}   :=    \Omega_\varphi^k(f,\delta)_{w,p}
 +  \overrightarrow \Omega_\varphi^k(f,\delta)_{w,p} + \overleftarrow\Omega_\varphi^k(f,\delta)_{w,p} ,
\ee
where
\[
\overrightarrow \Omega_\varphi^k(f,\delta)_{w,p} := \sup_{0<h\le 2k^2\delta^2 }
\|w \overrightarrow\Delta_h^k (f)\|_{\Lp[-1,-1+2k^2\delta^2]}
\]
and
\[
\overleftarrow\Omega_\varphi^k(f,\delta)_{w,p} :=  \sup_{0<h\le 2k^2\delta^2}\|w \overleftarrow\Delta_h^k (f)\|_{\Lp[1-2k^2\delta^2,1]} .
\]
If $\alpha=\beta=0$, then
$ \omega_\varphi^k(f,\delta)_{1,p} $ is equivalent to the usual DT modulus $\omega_\varphi^k(f,\delta)_p = \sup_{0<h\le \delta}\|\Delta_{h\varphi}^k(f)\|_{p}$.

It is easy to see that $\Omega_\varphi^k(f,\delta)_{\wab,p} \leq c \norm{\wab f}{p}$ for all $\alpha,\beta \in\R$. (Throughout this paper, $c$ denote positive constants that may be different even if they appear in the same line.) At the same time,
moduli $\omega_\varphi^k(f,\delta)_{\wab,p}$ are usually defined with the restriction $\alpha,\beta \geq 0$ for all $p \leq \infty$ and not just for $p=\infty$.  The reason for this is that, on one hand, $\omega_\varphi^k(f,\delta)_{\wab,p}\leq c \norm{\wab f}{p}$ if $\alpha,\beta\geq 0$, and, on the other hand, if $\alpha <0$ or $\beta<0$, then there are functions $f$ in $\Wpab$ for which $\omega_\varphi^k(f,\delta)_{\wab,p} = \infty$. Indeed, suppose that $p<\infty$ and that $\delta >0$ is fixed. If $f(x) :=  (x+1-\e)^{-\alpha-1/p} \chi_{[-1+\e, -1+2\e]}(x)$ with  $\alpha<0$ and $0<\e<2k^2\delta^2$, then
$\norm{\wab f}{p} \leq c $,  $\norm{\wab f(\cdot +\e)}{p} = \infty$, and $\norm{\wab f(\cdot +i\e)}{p} = 0$, $2\leq i\leq k$, and so
$\overrightarrow \Omega_\varphi^k(f,\delta)_{\wab,p} = \infty$.

If $\alpha,\beta \geq 0$, then it is easy to see that, if $f\in\Wpab$, $1\leq p<\infty$, then $\lim_{\delta\to 0^+}\omega_\varphi^k(f,\delta)_{\wab,p} = 0$. In the case $p=\infty$, the fact that $f$ is in $\W_\infty^{\alpha,\beta}$ implies that $\omega_\varphi^k(f,\delta)_{\wab,\infty}$ is bounded but it is not enough   to guarantee its convergence  to zero if $\alpha^2+\beta^2 \ne 0$ even if $f$ is continuous on $(-1,1)$
(consider, for example, $f(x)=  \wab^{-1}(x)$).
One can show (see \eg \cite[p. 287]{dsh} for a similar proof)  that,  if   $\alpha> 0$ and $\beta>0$, then  for $f\in\C(-1,1)$,
$\lim_{\delta\to 0^+} \omega_\varphi^k(f,\delta)_{\wab,\infty} =0$ iff $\lim_{x\to\pm 1} \wab(x)f(x)=0$.

 One can easily show that, for $\alpha,\beta\in\R$,
  \be \label{aaaomega}
 \sup_{f  \in\Bpab} \Omega_\varphi^k(f,\delta)_{\wab,q} \sim 1, \quad 1\leq q \leq p \leq \infty .
 \ee
 (Here and later in this paper, we write $F\sim G$ iff there exist positive constants $c_1$ and $c_2$  such that $c_1 F \leq G\leq c_2 F$. These constants are always independent of $\delta$, $n$ and $x$ but may depend on $k$, $\alpha$, $\beta$, $p$ and $q$.)
Indeed,
since $\Omega_\varphi^k(f,\delta)_{\wab,q} \leq c \norm{\wab f}{q}$, H\"{o}lder's inequality implies the upper estimate. The lower estimate follows, for example, from the fact that, for
 $k\in\N$, $\alpha,\beta \in\R$,   $0<p, q \leq \infty$, and $0<\delta\leq 1/(2k)$,
 the function
 \[
f_\delta(x) :=
\begin{cases}
(-1)^i  , & \mbox{\rm if} \quad \ds x \in \left[  k\delta i,  k\delta(i+1/2)\right] , \quad 0\leq i \leq  \lfloor 1/(2k\delta) \rfloor ,   \\
0, & \mbox{\rm otherwise,}
\end{cases}
\]
satisfies $\norm{\wab f_\delta }{p} \sim 1$   and
$\Omega_\varphi^k(f_\delta,\delta)_{\wab, q}   \geq c >0$ (see \lem{auxlemmatwox} for details).

The restriction $q\leq p$ in \ineq{aaaomega} is essential since
\[
 \sup_{f  \in\Bpab} \Omega_\varphi^k(f,\delta)_{\wab,q} = \infty, \quad \mbox{\rm if} \quad p<q .
 \]
This, of course, is expected since $\Wpab \not\subset \Wqab$, if $p<q$, and
 follows, for example, from \cor{auxcorone}.

 If $\alpha,\beta \geq 0$, then
   \be\label{aaaw}
 \sup_{f  \in\Bpab} \omega_\varphi^k(f,\delta)_{\wab,q} \sim 1 , \quad 1\leq q \leq p \leq \infty .
 \ee
 This follows from \ineq{aaaomega} and the observation that, for $\alpha,\beta \geq 0$, $\overrightarrow \Omega_\varphi^k(f,\delta)_{\wab,q} \leq  c \norm{\wab f}{q}$ and
 $\overleftarrow \Omega_\varphi^k(f,\delta)_{\wab,q} \leq  c \norm{\wab f}{q}$.

 In this paper, we show that if the suprema in \ineq{aaaomega} and \ineq{aaaw} are taken over the subset of $\Bpab$ consisting of all $k$-monotone functions, then these quantities become significantly smaller. This will allow us to obtain the exact rates (in some sense) of polynomial approximation in the weighted $\Lq$-norm of $k$-monotone functions in $\Bpab$.

Recall that   $f:I\to\R$ is said to be $k$-monotone on $I$ if its $k$th divided differences $[x_0,\ldots,x_k;f]$ are nonnegative for all selections of $k+1$ distinct points $x_0, \ldots, x_k$ in $I$, and denote by
$\M^k$ the set of all $k$-monotone functions on $(-1,1)$. In particular, $\M^0$, $\M^1$ and $\M^2$ are the sets of all nonnegative, nondecreasing and convex functions on $(-1,1)$, respectively. Note that if $f\in\M^k$, $k\geq 2$, then, for all $j\leq k-2$,
$f^{(j)}$ exists on $(-1,1)$   and is in $\M^{k-j}$.  In particular,   $f^{(k-2)}$ exists, is
convex, and therefore satisfies a Lipschitz condition on any closed
subinterval of $(-1,1)$, is absolutely continuous
on that subinterval, is continuous on $(-1,1)$, and  has left and right
(nondecreasing) derivatives, $f_-^{(k-1)}$ and $f_+^{(k-1)}$ on
$(-1,1)$. We also note that it is essential that   $(-1,1)$ and not $[-1,1]$ is used in the definition of $\M^k$ since the set of all $k$-monotone functions on the closed interval $[-1,1]$ contains only bounded functions
 (if $k\in\N$).

Our main result is

\begin{theorem} \label{maintheorem}
Let $k\in\N$,   $1\leq q <p\leq \infty$,   $\alpha,\beta \in J_p$, and $0<\delta< 1/4$. Then,
\be \label{inequalitym}
\sup_{f\in\Bpab \cap \M^k} \omega_\varphi^k(f,\delta)_{\wab,q} \sim
\left\{
\begin{array}{ll}
 \delta^{2/q-2/p}\,, &   \mbox{\rm if  $k\geq 2$ and $(k,q,p)\ne (2, 1, \infty),$}\\
\delta^{2}| \ln \delta |\,, &   \mbox{\rm if $k = 2$, $q=1$, $ p=\infty $, and $(\alpha,\beta)\ne (0,0)$, }\\
\delta^{2}   \,, &   \mbox{\rm if $k = 2$, $q=1$, $ p=\infty $, and $(\alpha,\beta)= (0,0)$, }\\
\delta^{2/q-2/p}\,, &   \mbox{\rm if   $k=1$ and $p < 2q$, }\\
 \delta^{1/q} \,, &  \mbox{\rm if $k=1$ and $p>2q$.}
\end{array} \right.
\ee
If $k=1$ and $p=2q$, then
\be \label{peq2q}
 c { \delta^{1/q}| \ln \delta |^{1/(2q)} \over |\ln|\ln \delta||^{\lambda/(2q)} } \leq \sup_{f\in\B_{2q}^{\alpha,\beta} \cap \M^1} \omega_\varphi^1(f,\delta)_{\wab,q} \leq c \delta^{1/q}  | \ln \delta |^{1/(2q)}, \quad \lambda > 1.
\ee
\end{theorem}

\begin{remark}
It is easy to see (and follows from Lemmas~\ref{auxnorms}, \ref{lemreduct} and \cor{corineq1}) that, for $k\in\N$,   $1\leq q \leq p\leq \infty$,   $\alpha,\beta \in J_p$, and $f\in \Wpab \cap \M^k$,
\[
\omega_\varphi^k(f,\delta)_{\wab,q} \leq c \norm{\wab f}{p}, \quad \delta > 0.
\]
Hence,  \thm{maintheorem} needs to be proved only for ``small'' $\delta$, and the restriction $\delta<1/4$ is chosen   for convenience only (to guarantee that none of the quantities in \ineq{inequalitym} and \ineq{peq2q} are zero while keeping them simple).
\end{remark}

In the case $\alpha=\beta=0$, all upper estimates and several lower estimates of \thm{maintheorem} were proved in \cite{k2009}, and so the upper estimates in \ineq{inequalitym} and \ineq{peq2q}   will only have to be established for $(\alpha,\beta)\ne (0,0)$ in the current paper.
 We remark that the fact that   the case $k=2$, $q=1$ and $p=\infty$ turned out to be anomalous for $(\alpha,\beta)\ne (0,0)$
causes rather significant difficulties in the proof  of \thm{maintheorem} for $k\geq 2$,  $q>1$ and $p=\infty$, since the rather simple main approach from \cite{k2009} can no longer be used.
(Section~\ref{newsectionpinf}  is devoted to  overcoming  these difficulties.)
We also note that the restriction $\alpha,\beta\in J_p$ in \thm{maintheorem}  guarantees that  the classes $\Bpab \cap \M^k$ contain   constants and so are rather rich. Without this restriction, we would have to deal with various anomalous situations and vacuous statements of theorems. For example, $\Bpab \cap \M^1 = \emptyset$ if $\alpha,\beta \leq -1/p$ since, in this case,   it is clear that $\Wpab \cap \M^1$ contains only  functions which are identically $0$ on $(-1,1)$. Similarly, it is possible to show that
$\Bpab \cap \M^2 = \emptyset$ if $\alpha,\beta \leq -1/p-1$. At the same time, putting restrictions on $\alpha$ and $\beta$ in the statements of some of our theorems would be a red herring
(\lem{lemineq1}, for example, is an illustration of this). Hence, an interested reader should keep in mind that even if a statement is given for all $\alpha,\beta \in \R$, it {\em may} happen that it only applies to trivial functions if  $\alpha,\beta\not\in J_p$.

It is convenient to denote
\be  \label{ups}
\Upsilon_\delta^{\alpha,\beta}(k,q,p)  :=
\left\{
\begin{array}{ll}
 \delta^{2/q-2/p}\,, &   \mbox{\rm if  $k\geq 2$, and $(k,q,p)\ne (2, 1, \infty),$}\\
\delta^{2}| \ln \delta |\,, &   \mbox{\rm if $k = 2$, $q=1$, $ p=\infty $, and $(\alpha,\beta)\ne (0,0)$ }\\
\delta^{2}   \,, &   \mbox{\rm if $k = 2$, $q=1$, $ p=\infty $, and $(\alpha,\beta)= (0,0)$, }\\
\delta^{2/q-2/p}\,, &   \mbox{\rm if   $k=1$ and $p < 2q$, }\\
 \delta^{1/q}  | \ln \delta |^{1/(2q)}   \,, &   \mbox{\rm if $k=1$ and $p=2q$,}\\
 \delta^{1/q} \,, &  \mbox{\rm if $k=1$ and $p>2q$.}
\end{array} \right.
\ee

The following is an immediate corollary of
 \thm{maintheorem}.

\begin{corollary} \label{maincor}
Let $k\in\N$,   $1\leq q <p\leq \infty$,   $\alpha,\beta \in J_p$, $f\in\M^k \cap\Wpab$ and $0<\delta< 1/4$. Then,
\be \label{maincorin}
\omega_\varphi^k(f,\delta)_{\wab,q} \leq c \Upsilon_\delta^{\alpha,\beta}(k,q,p) \norm{\wab f}{p}  ,
\ee
where  $\Upsilon_\delta^{\alpha,\beta}(k,q,p)$ which is defined in \ineq{ups}  is  best possible in the sense that \ineq{maincorin} is no longer valid if one increases (respectively, decreases) any of the powers of $\delta$ (respectively, $|\ln \delta|$) in its definition.
\end{corollary}

\begin{remark}
The restriction $q < p$ in the statement of \thm{maintheorem} is essential since, if $p<q$, then \cor{auxcorone} implies that
\[
\sup_{f\in\Bpab \cap \M^k} \omega_\varphi^k(f,\delta)_{\wab,q} = \infty ,
\]
and, if $p=q$, then it is easy to see that
\[
\sup_{f\in\Bpab \cap \M^k} \omega_\varphi^k(f,\delta)_{\wab,p} \sim 1 .
\]
\end{remark}

Let $\Pn$ be the set of algebraic polynomials of degree at most $n$, and denote
\[
E_n(f)_{w,q} := \inf_{P\in\Pn}\norm{w(f-P)}{q}
\]
and
\[
\E (X, \Pn)_{w,q} := \sup_{f\in X} E_n(f)_{w,q} .
\]
It is rather well known that
\[
\E (\Bpab, \Pn)_{\wab,q} \sim  1, \quad 1\leq q\leq p \leq \infty.
\]
(This also follows from \ineq{aaaw}, \ineq{luther} and  Remark~\ref{lowpoly}.) At the same time, for the class of $k$-monotone functions from  $\Bpab$, we have the following result.

\begin{theorem}  \label{mpoly}
Let  $1\leq q < p\leq \infty$, $k\in\N$, and  $\alpha,\beta\geq 0$.   Then, for any $n\in\N$,
\be  \label{1.8}
\E (\M^k\cap \Bpab, \Pn)_{\wab,q} \sim
\begin{cases}
n^{-2/q+2/p}\,, &   \mbox{\rm if  $k\geq 2$  and $(k,q,p)\ne (2, 1, \infty),$}\\
n^{-2}   \,, &   \mbox{\rm if $k = 2$, $q=1$,   $ p=\infty $, and $\alpha=\beta=0$,}\\
n^{-\min\left\{ 2/q-2/p, 1/q\right\}     }\,, &   \mbox{\rm if   $k=1$ and $p \neq  2q$. }\\
\end{cases}
 \ee
If $k=2$, $q=1$, $p=\infty$ and $(\alpha,\beta)\neq (0,0)$, then
\be \label{1.9}
c n^{-2} \leq  \E(\M^2 \cap \B_\infty^{\alpha,\beta}, \Pn)_{\wab,1} \leq c n^{-2} \ln (n+1) .
\ee
If  $k=1$ and $p =2q $, then
\be \label{1.10}
c  n^{-1/q}  \leq \E(\M^1 \cap \B_{2q}^{\alpha,\beta}, \Pn)_{\wab,q}  \leq c  n^{-1/q}  [\ln (n+1)]^{1/(2q)} .
\ee
Additionally, if $q>1$, then for any $\e>0$,
\be \label{tmpin}
\limsup_{n\to \infty}  n^{1/q}  [\ln (n+1)]^{-1/(2q)+\e}    \E(\M^1 \cap \B_{2q}^{\alpha,\beta}, \Pn)_{\wab,q} = \infty .
\ee
 \end{theorem}

In the case $\alpha=\beta=0$, \ineq{1.8} and the lower estimate in \ineq{1.10} were proved by Konovalov, Leviatan and Maiorov in  \cite[Theorem 1]{klm}.
The upper estimate in \ineq{1.10} and \ineq{tmpin} improve corresponding estimates in \cite[Theorem 1]{klm} (considered there in the special case  $\alpha=\beta=0$).

We remark that it is an open problem if $\ln(n+1)$ in \ineq{1.9} can be replaced by a smaller quantity or removed altogether, and if $[\ln(n+1)]^{1/2}$ is necessary in \ineq{1.10} in the case $(k,q,p)=(1,1,2)$. Also, while it follows from \ineq{tmpin} that, in the case $q>1$, the quantity $[\ln (n+1)]^{1/(2q)}$ in \ineq{1.10} cannot be replaced by $[\ln (n+1)]^{1/(2q)-\e}$ with $\e>0$, the precise behavior of
$\E(\M^1 \cap \B_{2q}^{\alpha,\beta}, \Pn)_{\wab,q}$ is still unknown. (See Section~\ref{secseven} for more details.)

Finally, we mention that several other applications of \thm{maintheorem} are given in Section~\ref{sec7}.

\sect{``Truncated'' $k$-monotone functions}

For  $k\geq 1$, we denote
\[
\M_+^k := \left\{ f\in\M^k \st f(x) =0,   \quad\mbox{\rm for all } x\in (-1,0]     \right\}.
\]
Note that, if $f\in\M_+^k$, then $f^{(i)}(0) = 0$, $0\leq i \leq k-2$, and  $f_-^{(k-1)}(0) = 0$.

In this section, we prove that it is sufficient to consider   classes $\M_+^k$ instead of $\M^k$ in \thm{maintheorem}  (see \lem{lemreduct}). This will significantly simplify the proofs of upper estimates.

\begin{lemma} \label{auxnorms}
Let $k\in\N$, $1\leq p \leq \infty$, $\alpha,\beta \in J_p$,  and  $f\in\M^k \cap \Wpab$. Then
\[
\norm{\wab  T_{k-1}(f)  }{p} \leq c \norm{\wab f}{p} ,
\]
where
\be \label{taylor}
T_{k-1}(f,x) :=   f_-^{(k-1)}(0)    x^{k-1}/(k-1)! + \sum_{i=0}^{k-2}  f^{(i)}(0)   x^i/ i!.
\ee
\end{lemma}

\begin{proof}
It follows
from \cite[Lemma 3.7]{k2009} that
$\norm{T_{k-1}(f)}{\Lp[-1/2,1/2]} \leq c \norm{f}{\Lp[-1/2,1/2]}$.
Therefore, taking into account that $\norm{\wab  }{p} \sim 1$ and $\wab(x) \sim 1 $ on $[-1/2, 1/2]$, we have
\[
\norm{\wab  T_{k-1}(f)  }{p}     \leq  c \norm{  T_{k-1}(f)  }{\infty}     \leq c \norm{T_{k-1}(f)}{\Lp[-1/2,1/2]} \leq c \norm{f}{\Lp[-1/2,1/2]}
  \leq   c \norm{\wab f}{p} ,
\]
where we used the fact that, for any   $p_{k-1}\in\Poly_{k-1}$  and $I\subseteq J$,
\[
\norm{p_{k-1}(f)}{\LL_\infty(J)} \leq c    \norm{p_{k-1}(f)}{\Lp(I)} , \quad c=c\left(k, |I|,|I|/|J|\right),
\]
which follows, for example, from \cite[(4.2.10) and (4.2.14)]{dl}.
\end{proof}

The following lemma can be easily proved by induction.
\begin{lemma}
Let  $f\in\M^k$, $k\in\N$, be such that $f^{(i)}(0) = 0$, $0\leq i \leq k-2$, and $f_-^{(k-1)}(0) = 0$. Then
$f$ is $j$-monotone on $[0,1)$ and $(-1)^{k-j}f$ is $j$-monotone on $(-1,0]$, for all $0\leq j\leq k-1$.
\end{lemma}

\begin{corollary} \label{cormplus}
If $k\in\N$ and $f\in \M_+^k$, then $f\in \M_+^j$, for all $0\leq j \leq k-1$.
\end{corollary}

\begin{lemma} \label{lemreduct}
Let $k\in\N$,   $1\leq q <p\leq \infty$,   $\alpha,\beta \in J_p$, and $\delta>0$.
Then
\[
\sup_{f\in\Bpab \cap \M^k} \omega_\varphi^k(f,\delta)_{\wab,q} \sim \sup_{f\in\Bpab \cap \M_+^k} \omega_\varphi^k(f,\delta)_{\wab,q}
+ \sup_{f\in\Bpba\cap \M_+^k} \omega_\varphi^k(f,\delta)_{\wba,q} .
\]
\end{lemma}

\begin{proof}
First of all, it is clear that
\be \label{switch}
\sup_{f\in\Bpab \cap \M^k} \omega_\varphi^k(f,\delta)_{\wab,q} = \sup_{f\in\Bpba \cap \M^k} \omega_\varphi^k(f,\delta)_{\wba,q} .
\ee
This immediately follows from the observation that $f(x)\in\Bpab \cap \M^k$ iff $(-1)^k f(-x) \in \Bpba \cap\M^k$.

Now, the estimate
\begin{eqnarray*}
2 \sup_{f\in\Bpab \cap \M^k} \omega_\varphi^k(f,\delta)_{\wab,q} &=&  \sup_{f\in\Bpab \cap \M^k} \omega_\varphi^k(f,\delta)_{\wab,q} + \sup_{f\in\Bpba \cap \M^k} \omega_\varphi^k(f,\delta)_{\wba,q}\\
&\geq & \sup_{f\in\Bpab \cap \M_+^k} \omega_\varphi^k(f,\delta)_{\wab,q} + \sup_{f\in\Bpba\cap \M_+^k} \omega_\varphi^k(f,\delta)_{\wba,q}
\end{eqnarray*}
is obvious since $\M_+^k \subset \M^k$. To prove the estimate in the opposite direction, suppose that $k$, $\alpha$, $\beta$, $\delta$, $q$ and $p$ satisfy all conditions of the theorem, and let
$f$ be an arbitrary function from $\M^k \cap \Bpab$.
Denote
\[
f_1(x) := \left( f(x) - T_{k-1}(f,x) \right) \chi_{[0,1]}(x) \andd
f_2(x) := \left( f(x) - T_{k-1}(f,x) \right) \chi_{[-1,0]}(x) ,
\]
where $T_{k-1}(f)$ is the Maclaurin polynomial of degree $\leq k-1$ defined in \ineq{taylor}.
It is clear that $f_1(x)$ and $\tilde f_2 (x):= (-1)^k f_2(-x)$ are both in $\M_+^k$.
Taking into account that $f-T_{k-1}(f) = f_1+f_2$, $|f_1|+|f_2| = |f_1+f_2|$,
\[
\norm{\wab f_2}{p} = \norm{\wba \tilde f_2}{p}
\andd
\omega_\varphi^k(f_2,\delta)_{\wab,q} = \omega_\varphi^k(\tilde f_2,\delta)_{\wba,q} ,
\]
we have
\begin{eqnarray*}
\norm{\wab f_1 }{p} + \norm{\wba \tilde f_2}{p} & = & \norm{\wab f_1 }{p} + \norm{\wab   f_2}{p}
 \leq   c\norm{\wab \left(|f_1| + |f_2|\right) }{p} \\
 &=& c \norm{\wab \left( f-T_{k-1}(f) \right) }{p} \leq c \norm{\wab  f  }{p} \leq c ,
\end{eqnarray*}
where the second last inequality follows from \lem{auxnorms}.

Now, if neither $f_1$ nor $\tilde f_2$ is identically equal to $0$ on $(-1,1)$,
 using the fact that
\[
\norm{\wab f_1 }{p}^{-1} f_1\in \Bpab\cap\M_+^k \andd       \norm{\wba \tilde f_2}{p}^{-1} \tilde f_2 \in \Bpba\cap\M_+^k
\]
 we have
\begin{eqnarray*}
\omega_\varphi^k(f, \delta)_{\wab,q} & \leq & \omega_\varphi^k(f_1, \delta)_{\wab,q}   + \omega_\varphi^k(f_2, \delta)_{\wab,q}  =   \omega_\varphi^k(f_1, \delta)_{\wab,q} + \omega_\varphi^k(\tilde f_2,\delta)_{\wba,q} \\
& = & \norm{\wab f_1 }{p}\omega_\varphi^k\left(\norm{\wab f_1 }{p}^{-1}f_1, \delta\right)_{\wab,q}   + \norm{\wba \tilde f_2}{p}\omega_\varphi^k\left(\norm{\wba \tilde f_2}{p}^{-1}\tilde f_2, \delta\right)_{\wba,q}\\
&\leq& c \sup_{f\in\Bpab \cap \M_+^k} \omega_\varphi^k(f,\delta)_{\wab,q} + c \sup_{f\in\Bpba\cap \M_+^k} \omega_\varphi^k(f,\delta)_{\wba,q} .
\end{eqnarray*}
If $f_1$ or  $\tilde f_2$ is  identically zero, the estimate is obvious.
\end{proof}

\begin{lemma} \label{lemreducttwo}
Let $k\in\N$,   $1\leq q <p\leq \infty$,   $\alpha,\beta \in J_p$, $\gamma_1,\gamma_2\in\R$, and $0<\delta< 1/k$.
Then
\[
\sup_{f\in\Bpab \cap \M^k} \omega_\varphi^k(f,\delta)_{\wab,q} \sim \sup_{f\in\B_p^{\gamma_1, \beta} \cap \M_+^k} \omega_\varphi^k(f,\delta)_{w_{\gamma_1, \beta},q}
+ \sup_{f\in\B_p^{\gamma_2, \alpha} \cap \M_+^k} \omega_\varphi^k(f,\delta)_{w_{\gamma_2, \alpha},q} .
\]
\end{lemma}

\begin{proof}
The lemma immediately follows from \lem{lemreduct} and the observation that
\[
\wab(x) \sim w_{\gamma_1, \beta}(x)   \andd    \wba(x) \sim w_{\gamma_2, \alpha}(x)       , \quad -1/2\leq x\leq 1 ,
\]
\[
\norm{w \Delta_{h\varphi}^k(f)}{\Lq(S)} = \norm{w \Delta_{h\varphi}^k(f)}{\Lq(S\cap[-1/2,1])}, \quad 0<h\leq 1/k ,
\]
and
\[
\norm{w f}{\Lp(S)} =\norm{w f}{\Lp(S\cap[0,1])} ,
\]
for any $f$ which is identically $0$ on $[-1,0]$.
\end{proof}

 \sect{Auxiliary results and upper estimates for $q=1$}

The proof of the following proposition is elementary and will be omitted.

\begin{proposition} \label{prchange}
Let $0<\eta<1$. Then the following   holds.
\begin{enumerate}[(a)]
\item  \label{a}
If  $|\lambda| \leq \sqrt{2\eta} $, then
the function $x \mapsto x + \lambda\varphi(x)$ is increasing on
$[-1+\eta, 1-\eta]$ and has the inverse
$y \mapsto \psi (\lambda, y)$, where
\be \label{psi}
\psi (\lambda, y) :=   {y-\lambda\sqrt{1-y^2 +\lambda^2} \over  1+\lambda^2  }   .
\ee
\item  \label{b}
If    $|\lambda| \leq \sqrt{2\eta} $, then
\be \label{change}
\int_{-1+\eta}^{1-\eta} g (x) f\left( x + \lambda\varphi(x) \right)\, dx =
\int_{-1+\eta+\lambda\sqrt{2\eta-\eta^2}}^{1-\eta+\lambda\sqrt{2\eta-\eta^2}}
f(y)\, g \left( \psi(\lambda,y) \right) {\partial \psi (\lambda,y) \over \partial y}  \, d y  .
\ee
\item  \label{c}
If $|x| \leq 1/\sqrt{4\lambda^2+1}$, then  $\ds \frac 12 \leq  {\partial (x + \lambda\varphi(x)) \over \partial x}   \leq 2$.
In particular,
if $|\lambda| \leq \sqrt{\eta/2} $, then $\ds \frac 12 \leq  {\partial (x + \lambda\varphi(x)) \over \partial x}   \leq 2$  for $x\in [-1+\eta, 1-\eta]$, and hence
$\ds \frac 12 \leq  {\partial \psi(\lambda, y) \over \partial y}   \leq 2$  for $y \in [-1+\eta+\lambda\sqrt{2\eta-\eta^2},  1-\eta+\lambda\sqrt{2\eta-\eta^2}]$.
\item  \label{d}
If  $|x|\leq 1-\eta$, then $\varphi(x) \leq \sqrt{2/\eta}(1-|x|)$.
\item  \label{e}
If    $|\lambda| \leq \sqrt{\eta}/2 $ and $|x| \leq 1-\eta$, then $(1-x)/4 \leq 1-x+\lambda\varphi(x) \leq 2(1-x)$ and
$(1+x)/4 \leq 1+x+\lambda\varphi(x) \leq 2(1+x)$.
\end{enumerate}
\end{proposition}

We are now ready to prove the main auxiliary theorem which will yield upper estimates in \thm{maintheorem} for $q=1$.   In view of \lem{lemreducttwo} we consider $f\in\M_+^k\cap\W_1^{\beta,\beta}$ noting that while we could consider $f\in\M_+^k\cap\W_1^{0,\beta}$, the symmetry makes things more convenient. We also note that it is possible to use the same approach in order to prove this theorem for $f\in\M^k\cap\W_1^{\alpha,\beta}$, but the estimates become more cumbersome. Finally, recall that
$
w_{\beta,\beta}(x) = \varphi^{2\beta}(x) .
$

\begin{theorem} \label{estmain}
Let $k\in\N$,   $\beta \in\R$, $f\in \M_+^k \cap \W_1^{\beta,\beta}$, and $0<\delta\leq 1/(2k)$.
Then
\begin{eqnarray} \label{estest}
\omega_\varphi^k(f,\delta)_{\wbb,1}  & \leq& c     \norm{\wbb f}{\LL_1[1-3k^2 \delta^2,1]}   \\ \nonumber
  &&    +  c  \sup_{0<h\leq \delta} h^k \norm{ (1-y^2)^{-k/2} \wbb(y) f(y)}{\LL_1[0,1- 2k^2 h^2]}
 \,.
\end{eqnarray}
\end{theorem}

The following corollary   immediately follows by H\"{o}lder's inequality and the  fact  that,
for $1\leq p'\leq \infty$ (with $1/p'+1/p=1$),
\[
 \norm{(1-y^2)^{-k/2}}{\LL_{p'}[0, 1-2k^2h^2]} \leq
 c  \left\{
 \begin{array}{ll}
 h^{-k+2/p'}\,, & \;\; \mbox{\rm if $kp'>2$}\,,\\
| \ln h |^{1/p'}\,, & \;\; \mbox{\rm if $kp'= 2$}\,,\\
1 \,, & \;\; \mbox{\rm if $kp' <2$}\,.
\end{array}
\right.
 \]

\begin{corollary} \label{maincorcor}
Let $k\in\N$, $ \beta \in \R$, $1\leq p \leq \infty$, $f\in \M_+^k\cap \Wpbb$, and $0<\delta\leq 1/(2k)$. Then
\be  \label{ininmaincor}
\omega_\varphi^k(f,\delta)_{\wbb,1}
  \leq   c  \norm{\wbb f}{p}
\left\{
\begin{array}{ll}
\delta^{2-2/p}\,, & \quad \mbox{\rm if $k\geq 3$, or $k=2$ and $1\leq p<\infty$,} \\
 & \quad \mbox{\rm or $k=1$ and $1\leq p < 2$, }\\
\delta^{2}| \ln \delta |\,, & \quad \mbox{\rm if $k = 2$ and $ p=\infty $, }\\
\delta \sqrt{| \ln \delta |} \,, & \quad \mbox{\rm if $k=1$ and $p=2$,}\\
 \delta \,, & \quad \mbox{\rm if $k=1$ and $2 < p\leq \infty$.}
\end{array} \right.
\ee
\end{corollary}

\begin{remark} \label{remnum}
If $\beta = 0$ and $k$ is even, or if $\beta=-1/2$ and $k$ is odd, then estimates   \ineq{estest} and \ineq{ininmaincor}  can be improved (see \rem{aarem} and \cite[Theorem 3.2]{k2009}). In fact, if $\beta=-1/2$ and $k=1$, then
we have
$
\omega_\varphi^1(f,\delta)_{w_{-1/2,-1/2},1} \leq c \delta^{2-2/p} \norm{\wbb f}{p} ,
$
for all $1\leq p \leq \infty$ and $f\in \M_+^1\cap \W_p^{-1/2,-1/2}$, and not only for $1\leq p <2$ as  \ineq{ininmaincor} implies. However, this is not too exciting since, on one hand,  $\beta = -1/2$ is in $J_p$ only if $1\leq p <2$ and, on the other hand, if
$p\geq 2$ then the set $\M_+^1\cap \W_p^{-1/2,-1/2}$ consists only of    functions which are identically equal to $0$ on $(-1,1)$.
\end{remark}

\begin{remark} \label{remreduct}
\cor{maincorcor}, together with Lemmas~\ref{lemreduct} and \ref{lemreducttwo}, implies the upper estimates in \thm{maintheorem} in the case $q=1$ (except for  the case $\alpha=\beta=0$ when k$=2$ and $p=\infty$ which follows from \cite{k2009}).
\end{remark}

Now, if $f\in \M_+^k\cap \Wpbb$ is such that $f \equiv 0$ on $[0, 1-A\delta^2]$, for some  constant $0<A\leq \delta^{-2}$, then taking into account that
\[
\sup_{0<h\leq \delta} h^k \norm{(1-y^2)^{-k/2}}{\LL_{p'}[1-A\delta^2, 1-2k^2h^2]} \leq c(A, k, p) \delta^{2/p'} ,
 \]
we have another corollary of \thm{estmain}.

\begin{corollary} \label{seccor}
Let $k\in\N$, $ \beta\in\R$, $1\leq p \leq \infty$, $0<\delta \leq 1/(2k)$, and let $f\in \M_+^k\cap \Wpbb$ be such that $f(x)=0$ for $x\in [0,1-A\delta^2]$, for some positive constant $A\leq \delta^{-2}$. Then
\begin{eqnarray*}
\omega_\varphi^k(f,\delta)_{\wbb,1}
  &\leq&  c  \delta^{2-2/p} \norm{\wbb f}{p} ,
\end{eqnarray*}
where $c$ depends on $A$.
\end{corollary}

 \begin{proof}[Proof of \thm{estmain}]
Let $h\in (0, \delta]$ be fixed.
Taking into account that $f\in\M_+^k$,  $\Delta_{h\varphi(x)}^k (f,x)\geq 0$   and   \prop{prchange}{b} with $\eta =2k^2h^2$  and $\lambda_i :=(i-k/2)h$, $0\leq i\leq k$,  we have
 \begin{eqnarray} \label{ineq1}
\lefteqn{  \|\wbb\Delta_{h\varphi}^kf\|_{\LL_1[-1+2k^2h^2 ,1-2k^2h^2]} }\\ \nonumber
& = &
 \sum_{i=0}^k {k \choose i} (-1)^{k-i}  \int_{-1+2k^2h^2}^{1-2k^2h^2}  \wbb(x) f(x+\lambda_i\varphi(x)) \, dx \\ \nonumber
 & = &
 \sum_{i=0}^k {k \choose i} (-1)^{k-i} \int^{1-2k^2h^2 + (2i-k) kh^2 \sqrt{1-k^2h^2}}_{-1+2k^2h^2 + (2i-k) kh^2 \sqrt{1-k^2h^2}}
 \wbb(\psi(\lambda_i,y)) f(y) {\partial \psi (\lambda_i,y) \over \partial y} \, dy \\ \nonumber
 & = &
 \sum_{i=0}^k {k \choose i} (-1)^{k-i}  \left(
 \int_0^{1-2k^2h^2 -  k^2h^2 \sqrt{1-k^2h^2}}
   +
 \int^{1-2k^2h^2 + (2i-k) kh^2 \sqrt{1-k^2h^2}}_{1-2k^2h^2 -  k^2h^2 \sqrt{1-k^2h^2}}
 \right) \\ \nonumber
 && \wbb(\psi(\lambda_i,y)) f(y) {\partial \psi (\lambda_i,y) \over \partial y}\, dy \\ \nonumber
 & = :&
 \sum_{i=0}^k {k \choose i} (-1)^{k-i} \left(\I_c +\I_r\right)\,.
\end{eqnarray}

It follows from \prop{prchange}{e} that
\be \label{auxweight}
\wbb(x) \sim \wbb(x+\lambda\varphi(x)) , \quad \text{for}\quad |x|\leq 1-\eta \andd |\lambda|\leq \sqrt{\eta}/2 .
\ee
In particular, this implies that
\[
\wbb(\psi(\lambda,y))) \sim \wbb(y), \quad \text{for}\quad y\in [-1+\eta+\lambda\sqrt{2\eta-\eta^2}, 1-\eta+\lambda\sqrt{2\eta-\eta^2}] \andd |\lambda|\leq \sqrt{\eta}/2 .
\]

Hence, noting also that \prop{prchange}{c} implies that
   $\left|{\partial \psi(\lambda_i, y) / \partial y} \right| \leq 2$, for all $0\leq i\leq k$, we have
 \begin{eqnarray} \label{ineq2}
\left|\sum_{i=0}^k {k \choose i} (-1)^{k-i}  \I_r  \right|  &\leq &
c
  \int_{1-3k^2h^2}^{1-k^2h^2}   \left|\wbb(y) f(y)\right| \, dy
  \leq
  c  \norm{\wbb f}{\LL_1[1-3k^2 \delta^2, 1]}  .
\end{eqnarray}
Now,
 \[
\left| \sum_{i=0}^k {k \choose i} (-1)^{k-i}  \I_c   \right|
    =
\left|\int_0^{1-2k^2h^2 -  k^2h^2 \sqrt{1-k^2h^2}}
f(y) A_k(y,h)\, dy \right|
\leq
\int_0^{1-2k^2h^2} |f(y)| |A_k(y,h)|\, dy ,
\]
where
 \begin{eqnarray*}
A_k(y,h) & :=& \sum_{i=0}^k {k \choose i} (-1)^{k-i}\wbb(\psi(\lambda_i,y))  \widetilde \psi (\lambda_i,y)
\end{eqnarray*}
and
\[
\widetilde \psi(\lambda_i,y) := {\partial \psi (\lambda_i,y) \over \partial y} = {\lambda_i y + \sqrt{1-y^2 + \lambda_i^2} \over (1+\lambda_i^2) \sqrt{1-y^2 + \lambda_i^2} } .
\]
Suppose now that $y\in [0,1-2k^2h^2]$ is fixed and, for convenience, denote $\vartheta := \varphi(y)$. Then $\vartheta \geq \sqrt{3} kh$.

Note that
\[
A_k(y,h)  = \sum_{i=0}^k {k \choose i} (-1)^{k-i}  g_y(\lambda_i /\vartheta) = \Delta_{h/\vartheta}^k(g_y, 0)     ,
\]
where
\[
g_y(t) := \wbb(\psi(t\vartheta,y))   \widetilde \psi(t\vartheta,y) .
\]

Recall that, if $g^{(m)}$ is continuous on $[x-m\mu/2,x+m\mu/2]$, then for some $\xi \in (x-m\mu/2, x+m\mu/2)$,
\be \label{aux}
\Delta^m_\mu(g,x) = \mu^m g^{(m)}(\xi)\,.
\ee

Hence,
\be \label{aky}
|A_k(y,h)| = | \Delta_{h/\vartheta} ^k(g_y, 0) | \leq  h^k \vartheta^{-k}   \norm{ {d^k  \over dt^k}g_y(t) }{\C[-1/2, 1/2]} .
\ee

We now note that
\[
\varphi (\psi(t\vartheta, y)) =  \vartheta {  t y + \sqrt{1+t^2}  \over  1+t^2\vartheta^2 } =  { \vartheta  \over    \sqrt{1 +t^2} - t y }
\]
and
\[
\widetilde \psi(t\vartheta, y)    =   {t y  + \sqrt{1 +t^2} \over  (1+t^2\vartheta^2) \sqrt{1 +t^2}  } = { 1  \over  \sqrt{1 +t^2}- t y  } \cdot {1 \over \sqrt{1 +t^2} }
\]
and, in particular,
\[
\widetilde \psi(t\vartheta, y)  ={ \varphi (\psi(t\vartheta, y)) \over \vartheta \sqrt{1 +t^2}} .
\]
Therefore, recalling that $\wbb = \varphi^{2\beta}$ we have
\[
g_y(t)  =  {\varphi^{2\beta+1}(\psi(t\vartheta, y))\over \vartheta \sqrt{1 +t^2}}
  =   \vartheta^{2\beta} (1 +t^2)^{-\beta-1} \left( 1 - {ty \over\sqrt{1+t^2}}  \right)^{-2\beta-1}  .
\]

\begin{remark} \label{aarem}
If $G_y(t):=  (g_y(t) + (-1)^k   g_y(-t))/2$, then
$A_k(y,h) =   \Delta_{h/\vartheta}^k(G_y, 0)$.
If $\beta=-1/2$ and $k$ is odd, then $G_y$ is identically equal to $0$, and so $|A_k(y,h)| = 0$.
Also, if $\beta=0$ and $k$ is even, then
$G_y(t) = (1+t^2 \vartheta^2)^{-1}$,
and so $|G_y^{(k)}(t)| \leq c \vartheta^k$ and $|A_k(y,h)| \leq c h^k$. Hence, \ineq{estest} can be improved in these cases.
\end{remark}

Noting that $|t|y/\sqrt{1+t^2} < 1$, we have the following expansion into binomial series
\[
\left( 1 - {ty \over\sqrt{1+t^2}}  \right)^{-2\beta-1} = \sum_{i=0}^\infty {-2\beta-1 \choose i} (-1)^i {t^iy^i \over (1+t^2)^{i/2}} ,
\]
and so
\[
g_y(t) = \vartheta^{2\beta } \sum_{i=0}^\infty {-2\beta-1 \choose i} (-1)^i {t^i y^{i} \over (1+t^2)^{\beta+1+i/2}} .
\]

The derivatives of this series are uniformly convergent  on   $[-1,1]$ (to take a simple interval) because it can be easily seen
that, for $|t|\leq 1$,
\[
\left|{d^k \over dt^k} {t^i   \over (1+t^2)^{\beta+1+i/2}} \right|
   \leq
 c \sum_{j=0}^{k} \left|\left[\left({t  \over \sqrt{1+t^2}}\right)^i \right]^{(j)}\right|
 \leq
 c \sum_{j=0}^{\min\{i,k\}} (i+1)^j  \left({|t|  \over \sqrt{1+t^2}}\right)^{i-j}
   \leq
 c (i+1)^k 2^{-i/2 } .
\]

Hence, for $|t|\leq 1$,
\[
\left| {d^k \over dt^k} g_y(t)\right|  \leq c  \vartheta^{2\beta  } \sum_{i=0}^\infty \left|{-2\beta-1 \choose i}\right| (i+1)^k  2^{-i/2}  \leq c \vartheta^{2\beta}.
\]
Estimate \ineq{aky} now implies that
\[
|A_k(y,h)| \leq c h^k \vartheta^{2\beta-k} ,
\]
and so
\begin{eqnarray} \label{ineq3}
\left| \sum_{i=0}^k {k \choose i} (-1)^{k-i}  \I_c   \right|
& \leq &
c h^k \int_0^{1-2k^2h^2} (1-y^2)^{\beta-k/2} |f(y)| \, dy .
\end{eqnarray}
Together with \ineq{ineq1}, inequalities   \ineq{ineq2} and \ineq{ineq3}   imply that
\begin{eqnarray} \label{ineq4}
 \Omega_\varphi^k(f,\delta)_{\wbb,1} & = &
c    \norm{\wbb f}{\LL_1[1-3k^2 \delta^2,1]}   \\  \nonumber
  &&    +  c  \sup_{0<h\leq \delta} h^k \norm{ (1-y^2)^{-k/2} \wbb(y) f(y)}{\LL_1[0,1- 2k^2 h^2]} .
\end{eqnarray}
Finally, \lem{lemineq1} (that we prove in Section~\ref{section4} for all $q\geq 1$) with $q=1$, together with \ineq{ineq4}, implies \ineq{estest}.
\end{proof}

\sect{Upper estimates for $q >1$} \label{section4}

\begin{lemma}\label{lemineq1}
Let $k\in\N$,   $1\leq q < \infty$,  $\alpha,\beta\in\R$, and $f\in\M^k_+ \cap \Wqab$. Then
\[
\overleftarrow\Omega_\varphi^k(f,\delta)_{\wab,q}   \leq c \norm{\wab f}{\Lq[1- 2k^2 \delta^2,1]}.
\]
\end{lemma}

\begin{proof} \cor{cormplus} implies that $f$ is non-negative and non-decreasing on $[0,1]$ and so, for  any $0<h\le 2k^2\delta^2$, we have
\begin{eqnarray*}
\lefteqn{  \|\wab \overleftarrow\Delta_h^k (f)\|_{\Lq[1-2k^2\delta^2,1]}^q   \leq
c \int_{1-2k^2\delta^2}^1  \sum_{i=0}^k \left[{k \choose i}\right]^q \wab^q(x) |f(x-ih)|^q \, dx} \\
& & \leq   c \sum_{i=0}^k \int_{1-2k^2\delta^2}^1     \wab^q(x) |f(x)|^q \, dx
  \leq   c \norm{\wab f}{\Lq[1- 2k^2 \delta^2,1]}^q ,
\end{eqnarray*}
and it remains to take supremum over   $h\in (0, 2k^2\delta^2]$.
\end{proof}

By  H\"{o}lder's inequality, the following corollary is an immediate consequence of \lem{lemineq1}.

\begin{corollary}\label{corineq1}
Let $k\in\N$,   $1\leq q < p\leq \infty$, $\alpha,\beta\in\R$, and $f\in\M^k_+ \cap \Wpab$. Then
\[
\overleftarrow\Omega_\varphi^k(f,\delta)_{\wab,q}   \leq c \delta^{2/q-2/p} \norm{\wab f}{\Lp[1- 2k^2 \delta^2,1]} .
\]
\end{corollary}

\begin{lemma}  \label{lema}
Let   $1 < q< \infty$, $\alpha,\beta \in\R$,  and let $f\in\Wqab$   be nonnegative on $[-1,1]$. Then,
\[
\omega_\varphi^1(f,\delta)_{\wab, q} \leq c
\omega_\varphi^1(f^q,\delta)_{w_{q\alpha, q\beta}, 1}^{1/q}.
\]
\end{lemma}

\begin{remark}
If $f\in\M^1\cap \Wqab$, $1 < q < \infty$, is nonnegative on $[-1,1]$, then $f^q \in \M^1 \cap \W_1^{q\alpha,q\beta}$.
\end{remark}

\begin{proof}
Let $1 < q< \infty$, and let $f\in\Wqab$   be nonnegative on $[-1,1]$.
It was shown in the proof of \cite[Lemma 3.4]{k2009} (and is easy to see) that,
\[
\left| \Delta_\mu^1(f,x) \right|^q \leq   \left| \Delta_\mu^1(f^q,x) \right| , \quad \mu >0 .
\]
This implies
\begin{eqnarray*}
\lefteqn{ \Omega_\varphi^1(f,\delta)_{\wab, q}^q  =  \sup_{0 < h\leq \delta} \int_{-1+2h^2}^{1-2h^2} \left|\wab(x) \Delta_{h\varphi(x)}^1 (f ,x)\right|^q dx }\\
&&\leq  \sup_{0 < h\leq \delta} \int_{-1+2h^2}^{1-2h^2}  \wab^{q}(x) \left| \Delta_{h\varphi(x)}^1 (f^q ,x)\right|  dx
 =     \Omega_\varphi^1(f^q,\delta)_{w_{q\alpha, q\beta}, 1}
\end{eqnarray*}
and, similarly,
\begin{eqnarray*}
\lefteqn{ \overleftarrow\Omega_\varphi^1(f,\delta)_{\wab,q}^q
  =
\sup_{0<h\leq 2\delta^2} \int_{1-2\delta^2}^1 \left|\wab (x) |\overleftarrow\Delta_h^1(f,x)   \right|^q \, dx} \\
& & \leq
\sup_{0<h\leq 2\delta^2} \int_{1-2\delta^2}^1  \wab^q (x) \left|\overleftarrow\Delta_h^1(f^q,x)\right|  \, dx
 =
\overleftarrow\Omega_\varphi^1(f^q,\delta)_{w_{q\alpha, q\beta},1}  ,
\end{eqnarray*}
and, since $\overrightarrow\Omega_\varphi^1(f,\delta)_{\wab,q}$ can be estimated similarly,  the proof is complete.
\end{proof}

\begin{lemma} \label{lemb}
Let   $1 < q< \infty$, $\alpha,\beta\in\R$, and let $f\in\M^2 \cap \Wqab$ be nonnegative on $[-1,1]$. Then, $f^q \in \M^2\cap \W_1^{q\alpha,q\beta}$, and
\[
\omega_\varphi^2(f,\delta)_{\wab, q} \leq c
\omega_\varphi^2(f^q,\delta)_{w_{q\alpha, q\beta}, 1}^{1/q}.
\]
\end{lemma}

\begin{proof}
It was shown in the proof of \cite[Lemma 3.5]{k2009} that, for any nonnegative  convex function $f$,
\[
\left( \Delta_\mu^2(f,x) \right)^q \leq 2^{q-1}  \Delta_\mu^2(f^q,x)\,, \quad \mu >0 ,
\]
 and the rest of the proof is  analogous to that of \lem{lema}.
\end{proof}

Now, taking into account that, for a nonnegative $f$,  $\norm{w_{q\alpha,q\beta} f^q}{p/q}^{1/q} = \norm{\wab f}{p}$, and using Lemmas~\ref{lema}, \ref{lemb}  and \cor{maincorcor} (with $p/q$ instead of $p$) we get the following result.

\begin{corollary}  \label{coraux}
Let $k=1$ or $k=2$, $\beta \in \R$, $1 < q <p\leq \infty$, $f\in \M_+^k\cap \Wpbb$, and $0<\delta\leq 1/(2k)$. Then
\begin{eqnarray*}
\omega_\varphi^k(f,\delta)_{\wbb,q}
 &\leq&  c  \norm{\wbb f}{p}
\left\{
\begin{array}{ll}
\delta^{2/q-2/p}\,, &   \mbox{\rm if  $k=2$ and $p<\infty$, or $k=1$ and $p < 2q$, }\\
\delta^{2/q}| \ln \delta |^{1/q}\,, &   \mbox{\rm if $k = 2$ and $ p=\infty $, }\\
\delta^{1/q}  | \ln \delta |^{1/(2q)}   \,, &   \mbox{\rm if $k=1$ and $p=2q$,}\\
 \delta^{1/q} \,, &  \mbox{\rm if $k=1$ and $p>2q$. }
\end{array} \right.
\end{eqnarray*}
\end{corollary}

 Lemmas~\ref{lemreduct} and \ref{lemreducttwo} now imply upper estimates in \thm{maintheorem} for $k=1$ and $k=2$ and $q>1$ except for the case $(k,p) = (2,  \infty)$, which will be dealt with separately in the next section.

 We will now finish the proof of the upper estimates in the case $k\geq 3$.
 It follows from \cite[Theorem 6.2.5]{dt} that
\be \label{ktwo}
\Omega_\varphi^k(f,\delta)_{\wab,q} \leq c \Omega_\varphi^2(f,\delta)_{\wab,q} , \quad k\geq 3.
\ee
Now, suppose that $f\in\M_+^k \cap \Wpbb$, $k\geq 3$. \cor{cormplus} implies that $f\in\M_+^2$, and so using
\cor{corineq1} and \ineq{ktwo} we have
\[
\omega_\varphi^k(f,\delta)_{\wbb,q} \leq c \Omega_\varphi^2(f,\delta)_{\wbb,q} + \overleftarrow\Omega_\varphi^k(f,\delta)_{\wbb,q}
\leq c \omega_\varphi^2(f,\delta)_{\wbb,q} + \delta^{2/q-2/p} \norm{\wbb f}{p} .
\]
We have already proved that
\be \label{lastin}
\omega_\varphi^2(f,\delta)_{\wbb,q} \leq \delta^{2/q-2/p} \norm{\wbb f}{p},  \quad  f\in\M_+^2 \cap \Wpbb ,
 \ee
 in the case $q>1$ and $p<\infty$, and will prove it for $q>1$ and $p=\infty$ in the next section, and so   upper estimates of \thm{maintheorem} for $k\geq 3$ and $q>1$ now follow from Lemmas~\ref{lemreduct} and \ref{lemreducttwo}.

 Hence, in order to finish the proof of all upper estimates in \thm{maintheorem} it remains to prove \ineq{lastin} in the case $q>1$ and $p=\infty$. This is done in Section~\ref{newsectionpinf} (see \lem{corauxxxxxx}).

\sect{Improvement of estimates for convex functions if  $q>1$} \label{newsectionpinf}

For $n\in\N$, we define $t_i := \cos\left(i\pi/n\right)$, $0\leq i\leq n$, and  $I_i := [t_i, t_{i-1}]$, $1\leq i\leq n$. Recall that $(t_i)_{0}^n$ is the so-called Chebyshev partition of $[-1,1]$. Some of its properties are stated in the following proposition that can be verified by straightforward computations.

\begin{proposition} \label{cheb}
For each $n\in\N$,  the following statements are valid.
\begin{enumerate}[(a)]
\item \label{cha}
For  $2\leq i\leq n-1$ and $x\in I_i$, $2\varphi(x)/n \leq  |I_i|  \leq 5 \varphi(x)/n$,
and $2 n^{-2} \leq |I_1| = |I_n| \leq 5 n^{-2}$.
\item \label{chb}
 $|I_{j-1}|/3 \leq |I_i| \leq 3 |I_{j-1}|$, $2\leq i\leq n$.
\item \label{chc}
For any $n\in\N$, $1\leq j\leq n$ and $\lambda \leq 1/n$, $t_j + \lambda \varphi(t_j) \leq t_{j-1} - \lambda \varphi(t_{j-1})$.
\end{enumerate}

\end{proposition}

 \begin{lemma} \label{lem100}
 Let   $0<\delta < 1/100$, $\beta\in\R$, and let $f\in \M^2 \cap \W_q^{\beta,\beta}$, $1 < q<\infty$, be such that its restrictions to  $[-1,-1+100\delta^2]$ and $[1-100\delta^2, 1]$ are linear polynomials.
  Then
\[
\Omega_\varphi^2 (f, \delta)_{\varphi^{2\beta}, q} \leq c  \delta^{1/q-1} \Omega_\varphi^2(f,   \delta)_{\varphi^{2\beta  -1+1/q}, 1}  .
\]
\end{lemma}

\begin{proof}
First, note that, for $0<h\leq \delta$, if $|x|\geq 1-85\delta^2$ then   $|x|- h\varphi(x)  \geq 1-100\delta^2$, and so $\Delta_{h\varphi(x)}^2(f,x) = 0$ if
 $x \in [-1,-1+85\delta^2] \cup [1-85 \delta^2, 1]$.
Therefore,
\begin{eqnarray*}
\Omega_\varphi^2 (f, \delta)_{\wbb, q}^q & = & \sup_{0< h \leq \delta} \norm{\wbb \Delta_{h\varphi}^2(f)}{\Lq[-1+8h^2 ,1-8h^2]}^q
  \leq    \sup_{0< h \leq \delta} \norm{\wbb \Delta_{h\varphi}^2(f)}{\Lq[-1+85\delta^2 ,1-85\delta^2]}^q .
\end{eqnarray*}
Now, note that, for each $m\in\N$ and $n \geq 2m+1$, if $\eta \geq 5m^2/n^2$, then $[-1+\eta ,1-\eta] \subset [t_{n-m}, t_m]$.
Hence, if we let  $n:= \lfloor 1/\delta \rfloor$  then $[-1+85\delta^2 ,1-85\delta^2] \subset [t_{n-4}, t_4] = \bigcup_{i=5}^{n-4} I_i$,
and so
\[
\Omega_\varphi^2 (f, \delta)_{\wbb, q}^q
  \leq
\sup_{0<h\leq \delta} \sum_{i=5}^{n-4} \int_{I_i} |\wbb(x) \Delta_{h\varphi(x)}^2(f,x)|^q \, dx .
\]
Since $h\leq \delta\leq 1/n$,  \prop{cheb}{chc} implies that if $x\in I_i$, then $x\pm  h\varphi(x) \in \tI_i := [t_{i+1}, t_{i-2}]$.

Now,  for $5\leq i\leq n-4$, let $p_i$ be the linear
 polynomial   interpolating $f$ at the endpoints of $\tI_i$, and let $g_i := f-p_i$.
 If $x_0 \in \tI_i$ is such that $\norm{g_i}{\C(\tI_i)} = |g_i(x_0)|$ (recall that convex functions are continuous in the interior of their domains), using the fact that $g_i$ is convex (and so lies below its secant lines)  and is $0$ at the endpoints of $\tI_i$,  we get
 \[
 \frac 12 |\tI_i| \norm{g_i}{\C(\tI_i)} = \frac 12 |\tI_i| |g_i(x_0)| \leq \int_{\tI_i} |g_i(x)|\, dx ,
 \]
 and so
 \[
 \norm{f-p_i}{\C(\tI_i)} \leq 2 |\tI_i|^{-1}  \norm{f-p_i}{\LL_1(\tI_i)} , \quad 5\leq i\leq n-4.
 \]
Therefore, recalling that $\wbb = \varphi^{2\beta}$ and using the fact that $\wbb(x)\sim \wbb(t_i)$, $x\in I_i$, and \prop{cheb}{cha} we have
\begin{eqnarray*}
\Omega_\varphi^2 (f, \delta)_{\wbb, q}^q & \leq &  \sup_{0<h\leq \delta} \sum_{i=5}^{n-4} \int_{I_i} \left|\varphi^{2\beta} (x) \Delta_{h\varphi(x)}^2(f-p_i,x)\right|^q \, dx \\
& \leq & c
  \sum_{i=5}^{n-4}   \varphi^{2\beta q}(t_i) |I_i| \norm{f-p_i}{\C(\tI_i)}^q   \\
  & \leq & c
 \sum_{i=5}^{n-4}   \varphi^{2\beta q}(t_i) |I_i|^{1-q} \norm{f-p_i}{\LL_1(\tI_i)}^q   \\
   & \leq & c  \sum_{i=5}^{n-4}  n^{q-1}  \varphi^{2\beta q-q+1}(t_i)   \norm{f-p_i}{\LL_1(\tI_i)}^q   \\
 & \leq & c n^{q-1} \left( \sum_{i=5}^{n-4}    \varphi^{2\beta  -1+1/q}(t_i)   \norm{f-p_i}{\LL_1(\tI_i)} \right)^q ,
 \end{eqnarray*}
where, in the last estimate, we used the inequality $\sum |a_i|^q \leq (\sum |a_i|)^q$.

It follows from \cite[Theorem 1]{k-whitney} that
\[
\norm{f-p_{i}}{\LL_1(\tI_i)} \leq c \omega_2(f, |\ttI_i|, \ttI_i)_1 ,  \quad 5\leq i\leq n-4,
\]
where $\ttI_i := [t_{i+2}, t_{i-3}]$ (since $\tI_i$ is in the ``interior'' of $\ttI_i$), and $\omega_2(f, \mu, I)$ is the usual second modulus on $I$.
 Proposition~\ref{cheb}(\ref{cha},\ref{chb}) implies that  $n|\ttI_i|/\varphi(x) \sim 1$, $x\in\ttI_i$, and, in particular,
$|\ttI_i|/\varphi(x) \leq c_*/n$, for some absolute constant $c_*$.
Now,
 \cite[Lemma 7.2, p. 191]{pp} yields
\[
\omega_2(f, \mu, [a,b])_1 \leq {c  \over  \mu} \int_0^{\mu} \int_a^b |\Delta_h^2 (f, x, [a,b])|  \, dx \, dh ,
\]
and hence
\begin{eqnarray*}
\omega_2(f, |\ttI_i|, \ttI_i)_1  & \leq &  c \omega_2(f, |\ttI_i|/(2c_*), \ttI_i)_1 \\
& \leq &
{c  \over |\ttI_i|} \int_{\ttI_i} \int_0^{|\ttI_i|/(2c_*)}  |\Delta_h^2 (f, x, \ttI_i )|  \, dh \, dx   \\
& \leq & {c  \over |\ttI_i|} \int_{\ttI_i} \int_0^{|\ttI_i|/(2c_*\varphi(x))}  \varphi(x) |\Delta_{h\varphi(x)}^2 (f, x, \ttI_i )|  \, dh \, dx   \\
& \leq & c n \int_{\ttI_i} \int_0^{1/(2n)}    |\Delta_{h\varphi(x)}^2 (f, x )|  \, dh \, dx  .
 \end{eqnarray*}
Therefore,
\begin{eqnarray*}
\Omega_\varphi^2 (f, \delta)_{\wbb, q}^q  & \leq &
c n^{q-1} \left( \sum_{i=5}^{n-4}    \varphi^{2\beta  -1+1/q}(t_i)    n \int_{\ttI_i} \int_0^{1/(2n)}    |\Delta_{h\varphi(x)}^2 (f, x )|  \, dh \, dx \right)^q \\
& \leq &
c  n^{2q-1} \left( \int_0^{1/(2n)} \sum_{i=5}^{n-4} \int_{\ttI_i} \varphi^{2\beta  -1+1/q}(x)|\Delta_{h\varphi(x)}^2 (f, x )|   \, dx\, dh \right)^q \\
& \leq &
c  n^{q-1} \left(\sup_{0<h\leq 1/(2n)}      \int_{t_{n-1}}^{t_1}  \varphi^{2\beta  -1+1/q}(x)|\Delta_{h\varphi(x)}^2 (f, x )|   \, dx  \right)^q \\
& \leq &
c  n^{q-1} \left(\sup_{0<h\leq 1/(2n)}      \int_{-1+8h^2}^{1-8h^2}  \varphi^{2\beta  -1+1/q}(x)|\Delta_{h\varphi(x)}^2 (f, x )|   \, dx  \right)^q  \\
& \leq &
c  n^{q-1} \Omega_\varphi^2(f, 1/(2n))_{\varphi^{2\beta  -1+1/q}, 1}^q ,
 \end{eqnarray*}
and it remains to recall that $n = \lfloor 1/\delta \rfloor$ and so, in particular,  $1/(2n) < \delta \leq 1/n$.
\end{proof}

\begin{lemma}  \label{corauxxxxxx}
Let   $ \beta \in\R$, $1 < q < \infty$  and $f\in \M^2_+\cap \W_\infty^{\beta,\beta}$. Then
\[
\omega_\varphi^2(f,\delta)_{\wbb,q}
  \leq   c \delta^{2/q} \norm{\wbb f}{\infty} .
\]
\end{lemma}

\begin{proof} Let $0<\delta <1/100$, denote $x_0 := 1-100 \delta^2$, and define
\[
 f_1 (x) :=
\begin{cases}
f(x) , & \mbox{\rm if } x \leq x_0 , \\
f(x_0) + f'_+(x_0)(x-x_0) , & \mbox{\rm if } x_0 < x \leq 1 ,
\end{cases}
\]
 Clearly, $f_1 \in \M^2_+$ and, since $0\leq f_1 (x) \leq f(x)$, $x_0 \leq x \leq 1$, we conclude that $\norm{\wbb f_1 }{\infty} \leq \norm{\wbb  f}{\infty}$.
 Also,   $f_2:= f- f_1\in \M^2_+$ is such that $f_2(x) = 0$ if $x \leq x_0$ and $\norm{\wbb f_2 }{\infty} \leq   \norm{\wbb  f}{\infty}$, and so \lem{lemb} and \cor{seccor} imply that
\[
\omega_\varphi^2(f_2,\delta)_{\wbb,q}   \leq
c \omega_\varphi^2(f_2^q,\delta)_{w_{q\beta, q\beta},1}^{1/q}
 \leq   c \left(\delta^{2} \norm{w_{q\beta, q\beta}  f_2^q}{\infty}\right)^{1/q}  \leq c \delta^{2/q} \norm{\wbb f}{\infty} .
 \]
Now, since $\overleftarrow\Omega_\varphi^2(f_1,\delta)_{\wbb,q} = 0$, by \lem{lem100} and \thm{estmain} we have
\begin{eqnarray*}
\omega_\varphi^2(f_1,\delta)_{\wbb,q}  & = & \Omega_\varphi^2(f_1,\delta)_{\wbb,q}
  \leq   c  \delta^{1/q-1} \Omega_\varphi^2(f_1,   \delta)_{\varphi^{2\beta  -1+1/q}, 1}\\
  & \leq&  c  \delta^{1/q-1} \norm{\varphi^{2\beta  -1+1/q}  f_1}{\LL_1[1-12\delta^2, 1]} +  c  \delta^{1/q-1} \sup_{0<h\leq   \delta} h^2 \norm{\varphi^{2\beta  -3+1/q} f_1}{\LL_1[0, 1-8h^2]} \\
   & \leq&  c  \delta^{1/q-1} \norm{\varphi^{2\beta} f_1}{\infty} \norm{\varphi^{ -1+1/q}}{\LL_{1}[1-12\delta^2, 1]}  \\
   && + c  \delta^{1/q-1}\norm{\varphi^{2\beta } f_1}{\infty} \sup_{0<h\leq  \delta} h^2 \norm{\varphi^{-3+1/q}}{\LL_{1}[0, 1-8h^2]}  \\
   & \leq & c \delta^{2/q} \norm{\wbb f}{\infty} ,
 \end{eqnarray*}
where, in the last estimate, we used
\[
\norm{\varphi^{-\gamma}}{\LL_{1}[1-c \delta^2, 1]} \leq
 c
\delta^{-\gamma+2 } , \quad   \mbox{\rm if} \quad  \gamma   <2 ,
 \]
 and
\[
 \norm{\varphi^{-\gamma}}{\LL_{1}[0, 1-c h^2]} \leq
 c
 h^{-\gamma+2 } , \quad \mbox{\rm if} \quad   \gamma  >2 .
 \]
\end{proof}

Together with Lemmas~\ref{lemreduct} and \ref{lemreducttwo}, this now completes the proof of the upper estimate in \thm{maintheorem}  in the case $k=2$, $p=\infty$ and $q>1$.

\sect{Lower estimates of moduli}

The following lemma verifies the lower estimate in \ineq{aaaomega}.

\begin{lemma} \label{auxlemmatwox}
Let $k\in\N$, $\alpha,\beta \in\R$,   $0<p, q \leq \infty$, and $0<\delta\leq 1/(2k)$.
Then the function
\[
f_\delta(x) :=
\begin{cases}
(-1)^i  , & \mbox{\rm if} \quad \ds x \in  J_i  , \quad 0\leq i \leq  \lfloor 1/(2k\delta) \rfloor ,   \\
0, & \mbox{\rm otherwise,}
\end{cases}
\]
where $J_i := \left[  k\delta i,  k\delta(i+1/2)\right]$,
is such that $\norm{\wab f_\delta }{p} \sim 1$,  and
 \[
 \Omega_\varphi^k(f_\delta,\delta)_{\wab, q}   \geq c >0 .
\]
\end{lemma}

\begin{proof}
Since  $\bigcup_{i=0}^{\lfloor 1/(2k\delta)\rfloor}  J_i \subset [0, 3/4]$,
\[
\norm{\wab f_\delta }{p}^p \sim \sum_{i=0}^{\lfloor 1/(2k\delta)\rfloor} |J_i| = \left(\lfloor 1/(2k\delta)\rfloor +1 \right) k\delta/2 \sim 1.
\]
Now, note that, if $x\in J_i$ and $0<h\leq \delta$, then $x \pm kh \varphi(x)/2 \not\in \cup_{j\ne i} J_j$, and so
\begin{eqnarray*}
 \Omega_\varphi^k(f_\delta,\delta)_{\wab, q}^q & \geq & \sup_{0<h\le \delta}  \sum_{i=0}^{\lfloor 1/(2k\delta)\rfloor} \int_{D_i}   \wab^q(x) \, dx \sim \sup_{0<h\le \delta}\sum_{i=0}^{\lfloor 1/(2k\delta)\rfloor} |D_i| ,
\end{eqnarray*}
where
\[
D_i  := \left\{ x \st x+(k/2-1)h\varphi(x)  \leq k\delta i \leq x+kh\varphi(x)/2 \right\} .
\]
Since $|D_i| \sim h$, $0\leq i \leq \lfloor 1/(2k\delta)\rfloor$, we have
\[
 \Omega_\varphi^k(f_\delta,\delta)_{\wab, q}^q \geq c \delta \lfloor 1/(2k\delta)\rfloor \geq c .
\]
\end{proof}

\begin{remark} \label{lowpoly}
For each $n\in\N$, letting $k=1$  and $\delta := 1/(4n)$ in \lem{auxlemmatwox},   noting  that $f_\delta$ is positive on $n+1$ intervals and negative on $n$ intervals $J_i$, and that any polynomial of degree $\leq n$ can have at most $n$ sign changes on $[-1,1]$, we conclude that
\[
E_n(f_\delta)_{\wab, q} \geq  c (n \delta)^{1/q} \geq c >0 .
\]
This implies that, for any $\alpha,\beta\in\R$ and $0<p,q \leq \infty$,
\[
\E (\Bpab, \Pn)_{\wab,q} \geq c > 0.
\]

\end{remark}

The following result verifies the lower estimate in \ineq{inequalitym} in the case $k=1$ and $p>2q$. Its proof is elementary and will be omitted.
\begin{lemma} If $f(x) = \chi_{[0,1]}(x)$, $\alpha\in\R$ and $\beta\in J_p$,  then $f\in\M^1$, $\norm{\wab f}{p} \sim 1$, and
$\omega_\varphi^1(f,\delta)_{\wab, q} \sim \delta^{1/q}$, for any   $0<\delta < 1$.
\end{lemma}

\begin{lemma} \label{auxlemmaone}
Let $k\in\N$, $0<p, q \leq \infty$, $\alpha \in\R$, $\beta \in J_p$, $\delta>0$, and   $0<\e \leq \min\{2k^2\delta^2, 1\}$.
Then the function $f(x) :=  \lambda (x-1+\e)_+^{k-1}$, $\lambda := \e^{-k-\beta-1/p+1}$,
 is such that   $f\in\M^k$,   $\norm{\wab f }{p} \sim 1$,  and
 \[
 \omega_\varphi^k(f,\delta)_{\wab, q}   \geq c \e^{1/q-1/p} .
\]
\end{lemma}

\begin{proof} It is straightforward to check that $\norm{\wab f }{p} \sim 1$. Now,  since
$S_\e(h) := [1-\e, 1-\e+\min\{\e,h\}/2] \subset [1-2k^2\delta^2,1]$ and $\overleftarrow\Delta_h^k (f,x) = f(x)$, $x\in S_\e(h)$,
we have
\begin{eqnarray*}
\overleftarrow \Omega_\varphi^k(f,\delta)_{\wab, q}^q
 & = & \sup_{0<h\le 2k^2\delta^2}\norm{\wab \overleftarrow\Delta_h^k (f)}{\Lq[1-2k^2\delta^2,1]}^q
  \geq
\sup_{0<h\le 2k^2\delta^2}
\int_{S_\e(h)} |\wab(x) f(x)|^q \, dx \\
&\geq & c \sup_{0<h\le 2k^2\delta^2}
\int_{S_\e(h)} \e^{q\beta} \lambda^q (x-1+\e)^{kq-q} \, dx
 \geq   c \sup_{0<h\le 2k^2\delta^2}
\e^{q\beta} \lambda^q (\min\{\e,h\})^{kq-q+1} \\
& \geq & c \lambda^q \e^{q\beta+kq-q+1} .
\end{eqnarray*}
Therefore,
\[
\omega_\varphi^k(f,\delta)_{\wab, q}   \geq    \overleftarrow \Omega_\varphi^k(f,\delta)_{\wab, q} \geq c \e^{1/q - 1/p} .
\]
If $p$ and/or $q$ are $\infty$, the proof is similar.
\end{proof}

Since $\lim_{\e \to 0^+} \e^{1/q - 1/p} = \infty$ if $p<q$, we immediately get the following corollary.
\begin{corollary} \label{auxcorone}
Let $k\in\N$, $\alpha,\beta \in\R$,    $\delta>0$, and  $0<p<q \leq \infty$.   Then, for any $A>0$, there exists $f \in \Bpab\cap \M^k$ such that
 \[
  \omega_\varphi^k(f,\delta)_{\wab, q}     \geq A .
\]
\end{corollary}

This corollary confirms that the one cannot expect to get any useful upper estimates for the moduli   $\omega_\varphi^k$
(even restricting classes to $k$-monotone function) if $p<q$.

\begin{corollary}  \label{auxcortwoo}
Let $k\in\N$, $0<p, q \leq \infty$, $\alpha \in\R$, $\beta \in J_p$, $0<\delta \leq 1/(2k)$, and  $\e := 2k^2\delta^2$.
Then the function $f(x) :=  \lambda (x-1+\e)_+^{k-1}$, $\lambda := \e^{-k-\beta-1/p+1}$,
 is such that   $f\in\M^k$,   $\norm{\wab f }{p} \sim 1$,  and
 \[
 \omega_\varphi^k(f,\delta)_{\wab, q}   \geq c \delta^{2/q-2/p} .
\]
\end{corollary}
This corollary   verifies the lower estimates in \ineq{inequalitym} in the cases $k\geq 2$ and  $(k,q,p) \neq (2,1,\infty)$ (unless $\alpha=\beta=0$), and $k=1$ and $p<2q$.

The following lemma yields the lower estimate in \ineq{inequalitym} in the case $(k,q,p)= (2,1,\infty)$ and $(\alpha,\beta)\neq (0,0)$.

\begin{lemma}[Lower estimate  in the case $k=2$, $q=1$ and $p=\infty$]
Let $\beta > 0$ and $f(x) := (1-x)^{-\beta}$. Then $f\in\M^2\cap \B_\infty^{0,\beta}$ and, if $\delta <1/5$,
\[
\Omega_\varphi^2 (f, \delta)_{\wb, 1} \geq c  \delta^2 |\ln \delta | .
\]
\end{lemma}

\begin{proof} It is obvious that $f\in\M^2\cap \B_\infty^{0,\beta}$.
Using the fact that
\[
\Delta_{h \varphi(x)}^2(f,x) = h^2 \varphi^2(x) f''(\xi) , \quad \mbox{\rm for some } \xi \in (x- h\varphi(x) , x+ h\varphi(x) ) ,
\]
we have
\begin{eqnarray*}
\Omega_\varphi^2(f,\delta)_{\wb, 1}
& \geq & c
\int_0^{1-8 \delta^2}  (1-x)^\beta \delta^2 \varphi^2(x) |f''(\xi_x)|  \, dx ,
\end{eqnarray*}
 where $\xi_x\in (x- \delta\varphi(x) , x+ \delta\varphi(x) )$. Now, \prop{prchange}{e}  implies that
 \[
 1-\xi_k \sim 1-x \pm \delta\varphi(x)  \sim 1-x ,
 \]
and so
$|f''(\xi_x)|   \geq c (1-x)^{-\beta-2}$.
Therefore,
\[
\Omega_\varphi^2(f,\delta)_{\wb, 1}  \geq
c  \delta^2
\int_0^{1-8\delta^2}   (1-x)^{-1} \, dx
  \geq   c \delta^2 |\ln\delta| .
\]
\end{proof}

We conclude this section with the proof of the lower estimate in \ineq{peq2q}.

\begin{lemma}[Lower estimate  in the case $k=1$ and $p=2q$]\label{auxlemmanew}
Let   $1\leq q <\infty$, $p=2q$,    $\beta > -1/p$, $0<\delta <1/4$,  and $\lambda > 1$. Then there exists   a function $f \in \Bpbb \cap \M^1_+$ such that
\be \label{lowk1}
 \Omega_\varphi^1(f,\delta)_{\wbb,q}   \geq c {\delta^{1/q} |\ln \delta|^{1/(2q)} \over |\ln|\ln \delta||^{\lambda/(2q)}} .
\ee
\end{lemma}

\begin{proof} Let $n = 2^m$, where $m = \lfloor \log_2(1/\delta) \rfloor+1$, and note that $1/n < \delta \leq 2/n$.

Suppose that  $(f_i)_{1}^n$ is a non-increasing sequence of real numbers such that $f_i = 0$, for $i>n/2$.
Now, recalling that $t_i = \cos(i\pi/n)$, $0\leq i \leq n$, and $I_i = [t_i, t_{i-1}]$, $1\leq i\leq n$,  define
\[
f(x) := f_i, \quad t_i < x \leq t_{i-1}, \quad 1\leq i\leq n .
\]
 In other words, $f$ is a non-decreasing piecewise constant spline with knots at $t_i$'s which is identically equal to $0$ on $[-1,0]$, \ie $f\in\M_+^1$.

Now, using Proposition~\ref{cheb}, the fact that $2i/n \leq \varphi(t_i) \leq 4i/n$, $1\leq i\leq n/2$,  and denoting   $\sum := \sum_{i=1}^{n/2}$, we have
\begin{eqnarray*}
\norm{\wbb f}{p}^p  &=&  \sum  \int_{I_i} \varphi^{2\beta p}(x)  |f(x)|^p \, dx
 \leq  c  \sum   |I_i| \varphi^{2\beta p}(t_i)   f_i^p
  \leq   c n^{-1} \sum  \varphi^{2\beta p+1}(t_i)   f_i^p \\
&\leq &    c n^{-2\beta p - 2 } \sum    i ^{2\beta p +1}  f_i^p   .
\end{eqnarray*}
 Now, let
\begin{eqnarray*}
D_i(h)  &:=& \left\{ x \st x-h\varphi(x)/2 \leq t_i \leq x+h\varphi(x)/2   \right\}\\
 &= & \left[ {t_i-(h/2) \sqrt{1-t_i^2 + h^2/4} \over 1 + h^2/4}, {t_i+(h /2) \sqrt{1-t_i^2 + h^2/4} \over 1 + h^2/4} \right] , \quad 1\leq i \leq n-1 .
\end{eqnarray*}
We note that intervals $D_i(h)$, $1\leq i \leq n-1$, have the following properties:
\begin{itemize}
\item[(i)] if $0<h\leq 1/n$, then $D_i(h) \cap D_{i-1}(h) = \emptyset$ for all $2\leq i\leq n-1$;
\item[(ii)] if $0<h\leq 1/(2n)$, then $D_i(h) \subset [-1+2h^2, 1-2h^2]$ for all $1\leq i\leq n-1$;
\item[(iii)] $|D_i(h)| \geq h \varphi(t_i)/2$, $1\leq i \leq n-1$.
\end{itemize}

 In order to verify (i), we suppose that $D_i(h) \cap D_{i-1}(h) \neq \emptyset$.
 Then there is $x\in [t_i, t_{i-1}]$  such that $x-h\varphi(x)/2 \leq t_i$ and $x+h\varphi(x)/2  \geq t_{i-1}$. Then,
$t_{i-1} - h\varphi(x)/2 \leq x \leq t_i + h\varphi(x)/2$, which implies $t_{i-1} - h\varphi(x)/2  \leq t_i + h\varphi(x)/2$, and so
\[
t_{i-1}-t_i \leq h \varphi(x), \quad \mbox{\rm for some }\; x \in [t_i, t_{i-1}] .
\]
At the same time, it is known that $|I_i| := t_{i-1}-t_i$ satisfies $\rho_n(x) \leq |I_i|$, for any $1\leq i \leq n$ and $x\in [t_i, t_{i-1}]$, where $\rho_n(x) := \sqrt{1-x^2}/n + 1/n^2$ (see \eg \cite{dsh},   or this can be verified directly).
Therefore,
\[
h \varphi(x) \leq \varphi(x)/n < \rho_n(x) \leq t_{i-1}-t_i,
\]
for any $x\in [t_i, t_{i-1}]$, which is a contradiction.

In order to verify (ii), we note that, in the case $i=1$ (which implies (ii) for all $1\leq i\leq n-1$), (ii) follows from the observation that, if $x = 1-2h^2$, then $x - h \varphi(x)/2 > t_1 = \cos(\pi/n)$. This inequality is equivalent to
\[
\cos(\pi/n) < 1 - 2h^2 - h^2 \sqrt{1-h^2} \iff 2h^2 + h^2 \sqrt{1-h^2} < 2 \sin^2(\pi/(2n)),
\]
which is true since
\[
(2h^2 + h^2 \sqrt{1-h^2})/2 \leq 3h^2/2 \leq 3/(8n^2)
 \andd
 \sin^2(\pi/(2n)) \geq \left[ (2/\pi)\pi/(2n)\right]^2 = 1/n^2 .
 \]

Finally, (iii) immediately follows from
\[
|D_i(h)|  = {h \sqrt{1-t_i^2 + h^2/4} \over 1 + h^2/4} \geq {h \varphi(t_i)   \over 1 + h^2/4} \geq {h \varphi(t_i)   \over 2} .
\]


Therefore, letting $h:= 1/(2n)$
we have
\begin{eqnarray*}
\Omega_\varphi^1(f,1/n)_{\wbb,q}^q & \geq &
\int_{-1+2h^{2}}^{1-2h^{2}} \varphi^{2\beta q}(x) \left(   \Delta_{h\varphi }^1 (f , x) \right)^q \, dx
  \geq
\sum  \int_{D_i(h)} \varphi^{2\beta q}(x) \left(   \Delta_{h\varphi}^1 (f , x) \right)^q \, dx \\
& \geq &
c \sum  \int_{D_i(h)} \varphi^{2\beta q}(t_i)  \left(  f_i - f_{i+1} \right)^q \, dx
  \geq
c \sum   h  \varphi^{2\beta q+1}(t_i)  \left(  f_i - f_{i+1} \right)^q   \\
& \geq & c
 n^{-2\beta q - 2} \sum   i^{2\beta q+1} \left(  f_i - f_{i+1} \right)^q .
\end{eqnarray*}

Now,   define
\[
f_i :=
\begin{cases}
2^{ 2\beta (m-k) + 2(m-k)/p} \zeta_k^{1/p}, & \mbox{\rm if} \quad  2^k\leq i \leq  2^{k+1}-1,  \quad 0\leq k \leq m-2, \\
0, & \mbox{\rm if} \quad  i \geq  2^{m-1}. \\
\end{cases}
\]
where $(\zeta_k)$ is a non-increasing sequence to be chosen later. Observe that $\left( 2^{ -2\beta k -2k/p} \right)_k$ is non-increasing since $\beta>-1/p$.
Then,
\begin{eqnarray*}
\norm{\wbb f}{p}^p &\leq&  c \sum_{k=0}^{m-2} \sum_{i=2^k}^{2^{k+1}-1} i^{2\beta p +1} 2^{-2\beta kp - 2k} \zeta_k
 \leq   c \sum_{k=0}^{m-2}\zeta_k
\end{eqnarray*}
 and
\begin{eqnarray*}
\Omega_\varphi^1(f,2^{-m})_{\wbb,q}^q &\geq&  c 2^{-2\beta mq  - 2m} \sum_{k=0}^{m-2} 2^{2\beta k q + k} \left( 2^{2\beta(m-k)+2(m-k)/p} \zeta_k^{1/p} -
 2^{2\beta(m-k-1)+2(m-k-1)/p} \zeta_{k+1}^{1/p} \right)^q \\
 &\geq & c  2^{-m} \sum_{k=0}^{m-2}   \left( \zeta_k^{1/p} - 2^{-2\beta - 2/p} \zeta_{k+1}^{1/p} \right)^q \\
 & \geq & c  2^{-m} \left( 1 - 2^{-2\beta - 2/p}   \right)^q \sum_{k=0}^{m-2}  \zeta_k^{1/2} .
\end{eqnarray*}
Now, let $\zeta_k := (k+2)^{-1} (\ln (k+2))^{-\lambda}$, where $\lambda>1$. Then,
\[
\norm{\wbb f}{p}^p \leq c \sum_{k=0}^{\infty} (k+2)^{-1} (\ln (k+2))^{-\lambda} \leq c
\]
and
\[
\Omega_\varphi^1(f,2^{-m})_{\wbb,q}^q  \geq c 2^{-m} \sum_{k=0}^{m-2}   (k+2)^{-1/2} (\ln (k+2))^{-\lambda /2}
  \geq   c 2^{-m} m^{1/2} (\ln m )^{-\lambda /2 }.
\]
Finally, recalling that $2^{-m} <  \delta \leq 2^{1-m}$ and replacing $f$ with $g := \norm{\wbb f}{p}^{-1} f$ we get a function in $\Bpbb\cap \M_+^1$ such that
\[
 \Omega_\varphi^1(g,\delta)_{\wbb,q} \geq  \norm{\wbb f}{p}^{-1} \Omega_\varphi^1(f,2^{-m})_{\wbb,q}
   \geq c {\delta^{1/q} |\ln \delta|^{1/(2q)} \over |\ln|\ln \delta||^{\lambda/(2q)}}  .
\]
\end{proof}

\begin{remark}
One can improve the estimate \ineq{lowk1} slightly by letting
\[
 \zeta_k := (g_{m,\lambda}(c(k+1)))^{-1} ,
\]
where
\[
g_{m,\lambda}(x) := x (\ln x) (\ln \ln x) \dots (\underbrace{\ln \dots \ln}_{m}  x)  (\underbrace{\ln \dots \ln}_{m+1}  x)^\lambda ,
\]
with $m\in\N$, $\lambda > 1$ and
a sufficiently large constant $c=c(m)$ that guarantees that $g_{m,\lambda}$ is well defined on $[c, \infty)$.
\end{remark}

\sect{Proof of \thm{mpoly}} \label{secseven}

It was proved by Luther and Russo \cite[Corollary 2.2]{lr}  that, for $\alpha,\beta \geq 0$, there exists  $n_0\in\N$  such that
\be \label{luther}
E_n(f)_{\wab,q}\le c\omega^k_\varphi(f,n^{-1})_{\wab,q},\quad n\ge n_0 .
\ee
 If $\alpha=\beta=0$, then this is a well known Jackson type estimate that was proved by Ditzian and Totik in \cite[Theorem 7.2.1]{dt}.
Taking into account that, for $0\leq n < n_0$,
$E_n(f)_{\wab,q}   \leq c \norm{\wab f}{q} \leq c \norm{\wab f}{p}$,   if $q \leq p$,
we immediately get the following corollary of \thm{maintheorem} that implies all upper estimates in \thm{mpoly}.

\begin{corollary} \label{corrjackzero}
Let  $1\leq q < p\leq \infty$, $k\in\N$,   $\alpha,\beta\geq 0$, and let $f\in\M^k\cap \Wpab$.   Then, for any $n\in\N$,
\be \label{noder}
E_n(f)_{\wab,q} \leq c  \norm{\wab f }{p}
\begin{cases}
n^{-2/q+2/p}\,, &   \mbox{\rm if  $k\geq 2$, and $(k,q,p)\ne (2, 1, \infty),$}\\
n^{-2} \ln (n+1) \,, &   \mbox{\rm if $k = 2$, $q=1$,   $ p=\infty $, and  $(\alpha,\beta)\ne(0,0)$,}\\
n^{-2}   \,, &   \mbox{\rm if $k = 2$, $q=1$,   $ p=\infty $, and $\alpha=\beta=0$,}\\
n^{-2/q+2/p}\,, &   \mbox{\rm if   $k=1$ and $p < 2q$, }\\
n^{-1/q}  [\ln (n+1)]^{1/(2q)}   \,, &   \mbox{\rm if $k=1$ and $p=2q$,}\\
n^{-1/q} \,, &  \mbox{\rm if $k=1$ and $p>2q$.}
\end{cases}
 \ee
 \end{corollary}

A matching inverse result to \ineq{luther} is given by (see \cite[Theorem 8.2.4]{dt})
\be \label{dtinv}
\omega^k_\varphi(f,\delta)_{\wab,q} \leq c \delta^k \sum_{0\leq i < 1/\delta} (i+1)^{k-1} E_i(f)_{\wab,q} .
\ee

Since, for $\mu, \lambda \in \R$ and  $0<\delta<1/4$,
\[
\int_2^{1/\delta}  x^\mu (\ln x)^\lambda \, dx         \sim
\begin{cases}
1 , & \mbox{\rm if} \quad \mu< -1 ,\\
\ds  \delta^{-\mu-1} |\ln \delta|^\lambda  , & \mbox{\rm if} \quad \mu>-1 ,\\
1 , & \mbox{\rm if} \quad \mu=-1, \lambda<-1, \\
|\ln \delta|^{1+\lambda}  , & \mbox{\rm if} \quad \mu=-1, \lambda>-1 ,\\
\ln|\ln \delta| ,  & \mbox{\rm if} \quad \mu=-1, \lambda=-1 ,
\end{cases}
\]
estimate \ineq{dtinv}  implies, in particular, that if for a function $f\in\Bpab \cap \M^k$,
\[
E_n(f)_{\wab,q} \leq c (n+2)^{\mu-k+1} [\ln(n+2)]^\lambda , \quad n\in\N_0 ,
\]
 then
\[
\omega^k_\varphi(f,\delta)_{\wab,q} \leq c
\begin{cases}
\delta^k , & \mbox{\rm if} \quad \mu< -1 ,\\
\ds  \delta^{k-\mu-1} |\ln \delta|^\lambda  , & \mbox{\rm if} \quad \mu>-1 ,\\
\delta^k , & \mbox{\rm if} \quad \mu=-1, \lambda<-1, \\
\delta^k |\ln \delta|^{1+\lambda}  , & \mbox{\rm if} \quad \mu=-1, \lambda>-1 ,\\
\delta^k \ln|\ln \delta| ,  & \mbox{\rm if} \quad \mu=-1, \lambda=-1 .
\end{cases}
\]
Together with lower estimates in \thm{maintheorem} this implies that none of the powers of $n$ in \ineq{noder} can be decreased (except for some cases when $q=1$ and $k\leq 2$).
This is made  precise in Corollaries~\ref{corrr1} and \ref{corrr2} which imply lower estimates in \ineq{1.8}, \ineq{1.9} and \ineq{1.10}.

Whether or not powers of $\ln(n+1)$ in \ineq{noder} can be decreased is more involved. 
In the case $k=2$, $q=1$, $p=\infty$ and $(\alpha,\beta)\neq (0,0)$, we only know that
\[
c n^{-2} \leq  \sup_{f\in \M^2 \cap \B_\infty^{\alpha,\beta}} E_n(f)_{\wab,1} \leq c n^{-2} \ln (n+1)
\]
(see \cor{corrr2} with $r=0$ for the lower estimate), and so it is an open problem if $\ln(n+1)$ in this estimate can be replaced by $o(\ln(n+1))$ or removed altogether.

In the   case $k=1$ and $p=2q$, if $E_n(f)_{\wab,q} \leq c (n+2)^{-1/q} [\ln (n+2)]^\lambda$, $n\in\N_0$ (\ie $\mu = -1/q$), for any function  $f\in\M^1 \cap \Bpab$,  then
\[
\omega^1_\varphi(f,\delta)_{\wab,q} \leq c
\delta^{1/q } |\ln \delta|^\lambda  ,  \quad  \mbox{\rm if } \quad q>1 .
\]
Together with lower estimates of \thm{maintheorem}  this implies that, if $k=1$ and $p/2=q>1$,  then the quantity  $n^{-1/q}  [\ln (n+1)]^{1/(2q)}$ in \ineq{noder} cannot be replaced by
$n^{-1/q}  [\ln (n+1)]^{1/(2q)-\e}$, for any $\e>0$. Also, this yields \ineq{tmpin}.

If $k=1$, $q=1$ and $p=2$, then we   know that (see \cor{corrr1} with $k=1$ for the lower estimate)
\[
c n^{-1} \leq  \sup_{f\in \M^1 \cap \B_2^{\alpha,\beta}} E_n(f)_{\wab,1} \leq c n^{-1} [\ln (n+1)]^{1/2} ,
\]
and it is an open problem if $[\ln(n+1)]^{1/2}$ in this estimate is necessary.

\sect{Other applications} \label{sec7}

{\bf \large 1.}
Let $1\leq p\leq \infty$, $r\in\N$.
Then
\[
\Wpabr  := \left\{ f: [-1,1]\mapsto\R  \st f^{(r-1)}\in \AC_\loc(-1,1) \andd \norm{\wab   f^{(r)}}{p} < \infty \right\} ,
\]
and for convenience denote $\W_{p,0}^{\alpha,\beta} := \Wpab$. Note that, if $\alpha=\beta=r/2$, then $\W_{p,r}^{r/2,r/2} = B_p^r$, the classes discussed in \cites{kls, kls-ca}.

The following lemma is a generalization of \cite[Lemma 3.4]{kls-ca}.

\begin{lemma} \label{hsp}
Let $1\leq p\leq \infty$, $r\in\N_0$, $\alpha,\beta\in\R$ and let $f\in\W_{p, r+1}^{\alpha,\beta}$. Then $f\in\W_{p,r}^{\alpha -\gamma, \beta-\gamma}$,
for any $\gamma <1$ such that $\alpha-\gamma, \beta-\gamma \in J_p$.
\end{lemma}

\begin{proof}
Given $f\in\W_{p, r+1}^{\alpha,\beta}$, taking into account that $\norm{\wg}{p} <\infty$ and replacing $f(x)$ with $f(x) - x^r f^{(r)}(0) /r!$ we can assume
  that $f^{(r)}(0)=0$.
Now, if $p=\infty$, then
\begin{eqnarray*}
\lefteqn{ \norm{   \wg f^{(r)}}{\infty}    \leq     \norm{\wg(x)  \int_0^x    f^{(r+1)}(u) \, du}{\infty}} \\
 & \leq&   \norm{ \wab  f^{(r+1)}}{\infty}
\norm{ \wg (x)  \int_0^x  \wab^{-1}(u) \, du}{\infty}
  \leq   c \norm{ \wab  f^{(r+1)}}{\infty}  .
\end{eqnarray*}
Similarly, if  $p=1$, then
\begin{eqnarray*}
\norm{ \wg f^{(r)}}{1}&=&\int_{-1}^1 \wg(x)
\left|\int_0^x f^{(r+1)}(u)\,du\right|\,dx\\
&\le&\int_{-1}^1  \wg(x) \left|\int_0^x
\wab(u)|f^{(r+1)}(u)| \wab^{-1}(u) \,du\right|\,dx\\
&\leq&
\norm{ \wab f^{(r+1)}}{1}\int_{-1}^1 \wg (x) \max_{u\in [0,x]} \wab^{-1}(u)  dx
\leq c\norm{
\wab  f^{(r+1)}}{1}.
\end{eqnarray*}
Suppose now that $1<p <\infty$ and denote $p' := p/(p-1)$.
Using H\"older's inequality we have
\begin{eqnarray*}
\norm{\wg f^{(r)}}{p}^p&=&\int_{-1}^1\wg^p(x)\left|\int_0^x f^{(r+1)}(u)\,du\right|^p\,dx \\
& \leq& \int_{-1}^1 \wg^p(x)\left|\left(\int_0^x \wab^{-p'}(u)\, du\right)^{1/p'}
\left(\int_0^x|\wab (u) f^{(r+1)}(u)|^p\,du\right)^{1/p}\right|^p\,dx\\
& \leq&
\norm{\wab f^{(r+1)}}{p}^p
\left( \int_{-1}^0 + \int_0^1 \right)  \wg^p(x)\left| \int_0^x \wab^{-p'}(u)\, du\right|^{p/p'}\,dx \\
& =:& \norm{\wab f^{(r+1)}}{p}^p \cdot \left( I_{\alpha,\beta,\gamma}^- + I_{\alpha,\beta,\gamma}^+ \right) .
\end{eqnarray*}
We will now show that $I_{\alpha,\beta,\gamma}^+ \leq c$ (the proof that the same estimate holds for $I_{\alpha,\beta,\gamma}^-$ is analogous). Indeed,
if $\beta p' \ne 1$, then
\begin{eqnarray*}
I_{\alpha,\beta,\gamma}^+ & \leq & c \int_0^1 (1-x)^{(\beta - \gamma)p} \left( \int_0^x (1-u)^{-\beta p'} \, du \right)^{p/p'}\, dx \\
& \leq & c \int_0^1 (1-x)^{(\beta - \gamma)p} \left(\max\{ 1, (1-x)^{-\beta p' +1} \}\right)^{p/p'}\, dx \\
&\leq &
c \int_0^1  \max\left\{ (1-x)^{(\beta - \gamma)p}, (1-x)^{-\gamma p + p -1} \right\} \, dx \leq c .
\end{eqnarray*}
Finally, if $\beta p' = 1$ (and so $\beta = 1-1/p$), then
\[
I_{\alpha,\beta,\gamma}^+   \leq   c \int_0^1 (1-x)^{(\beta - \gamma)p} |\ln (1-x) |^{p/p'} \, dx
  \leq   c \int_0^1 (1-x)^{p(1-\gamma)-1} |\ln (1-x) |^{p-1} \, dx \leq c .
\]
This completes the proof.
\end{proof}

\begin{remark} \label{rrrmrk}
 We actually proved that,
 if  $f\in\W_{p, r+1}^{\alpha,\beta}$ is such that    $f^{(r)}(0)=0$, then
\[
\norm{\wg f^{(r)}}{p} \leq c  \norm{\wab  f^{(r+1)}}{p}
\]
provided that $\gamma <1$ and $\alpha-\gamma, \beta-\gamma \in J_p$.

\end{remark}

\begin{corollary}
Let $1\leq p\leq \infty$, $r\in\N_0$ and $\alpha,\beta \in J_p$.
Then
\[
\W_{p, r+1}^{\alpha+(r+1)/2, \beta+(r+1)/2} \subset \W_{p, r}^{\alpha+r/2, \beta+r/2}
\]
and, in particular,
\[
\W_{p, r}^{\alpha+r/2, \beta+r/2} \subset \Wpab .
\]
\end{corollary}

It was shown in \cite[Theorem 5.1]{kls} that, if $1\leq q\leq \infty$, $0<r<k$, and $f$ is such that
$f^{(r-1)}$ is locally absolutely continuous in $(-1,1)$ and
$\wab \varphi^r f^{(r)}\in \Lq[-1,1]$, $\alpha,\beta\geq 0$, then
\be\label{derivative}
\omega_\varphi^k(f,\delta)_{\wab,q}\leq
ct^r\omega^{k-r}_\varphi(f^{(r)},\delta)_{\wab\varphi^r,q} . \ee

Taking into account that $\wab\varphi^r =  w_{\alpha+r/2, \beta+r/2}$, together with \ineq{luther}, this implies the following Jackson-type result for weighted polynomial approximation (see also \cite[Theorem 5.2]{kls}).

\begin{corollary} \label{dircor}
If $k\in\N$, $0\leq r \leq k-1$, $1\leq q\leq \infty$, $\alpha,\beta\geq 0$, and  $f \in \W_{q, r}^{\alpha+r/2,\beta+r/2}$,   then there exists $n_0 \in\N$ such that
\be \label{direct}
E_n(f)_{\wab,q}\le c n^{-r} \omega^{k-r}_\varphi(f^{(r)},n^{-1})_{w_{\alpha+r/2, \beta+r/2},q} , \quad n\geq n_0 .
\ee

\end{corollary}

Now, let $1\leq q < p\leq \infty$, $k\in\N$, $1\leq r \leq k-1$, and let $f\in\M^k\cap  \W_{p, r}^{\alpha+r/2,\beta+r/2}$.
Using \cor{maincor} and the fact that $f^{(r)}\in\M^{k-r}$, we conclude that, for $n\geq n_0$,
\begin{eqnarray} \label{drdr}
E_n(f)_{\wab,q}
& \leq&  c n^{-r} \Upsilon_{1/n}^{\alpha+r/2,\beta+r/2}(k-r, q, p) \norm{w_{\alpha+r/2, \beta+r/2} f^{(r)}}{p} .
 \end{eqnarray}

It is not hard to see that this estimate holds for $r-1 \leq n < n_0$ as well. Indeed, given a function
$f\in \W_{p, r}^{\alpha+r/2,\beta+r/2}$, let $T_{r-1}(f)$ be its Maclaurin polynomial of degree $\leq r-1$ (see \ineq{taylor}).
Then, for $r-1 \leq n < n_0$,  we have using \rem{rrrmrk}
\[
E_n(f)_{\wab,q} \leq \norm{\wab(f-T_{r-1}(f))}{q}   \leq c \norm{\wab \varphi^r f^{(r)}}{q} \leq c  \norm{w_{\alpha+r/2, \beta+r/2} f^{(r)}}{p} ,  \quad q \leq p .
\]

Hence,  the following  is another corollary of \thm{maintheorem}.

\begin{corollary} \label{corrjack}
Let  $1\leq q < p\leq \infty$, $k\geq 2$, $1\leq r \leq k-1$, $\alpha,\beta\geq 0$, and let $f\in\M^k\cap \W_{p, r}^{\alpha+r/2,\beta+r/2}$.   Then, for any $n\geq r$,
\[
E_n(f)_{\wab,q} \leq c  \norm{w_{\alpha+r/2, \beta+r/2} f^{(r)}}{p}
\begin{cases}
n^{-r-2/q+2/p}\,, &   \mbox{\rm if  $k-r\geq 2$  and $(k-r,q,p)\ne (2, 1, \infty),$}\\
n^{-r-2} \ln (n+1) \,, &   \mbox{\rm if $k -r = 2$, $q=1$ and  $ p=\infty $,}\\
n^{-r-2/q+2/p}\,, &   \mbox{\rm if   $k-r=1$ and $p < 2q$, }\\
n^{-r-1/q}  [\ln (n+1)]^{1/(2q)}   \,, &   \mbox{\rm if $k-r=1$ and $p=2q$,}\\
n^{-r-1/q} \,, &  \mbox{\rm if $k-r=1$ and $p>2q$.}
\end{cases}
 \]
 \end{corollary}

It follows from Corollaries~\ref{corrr1} and \ref{corrr2} that   estimates in \cor{corrjack} are exact in the sense that none of the powers of $n$ can be decreased.
Using the inverse theorem \cite[Theorem 9.1]{kls-ca} it is also possible to show that, in the case $\alpha=\beta=0$, $k=r+1$ and $p/2=q>1$, the power $1/(2q)$ of $\ln(n+1)$ cannot be decreased.

\vspace{1cm}
{\bf \large 2.}
 Littlewood's inequality
$\norm{g}{q} \leq \norm{g}{s}^{\theta} \norm{g}{p}^{1-\theta}$, $1/q=\theta/s + (1-\theta)/p$, $1\leq s < q <p \leq \infty$,
implies that
\[
\Omega_\varphi^k(f,\delta)_{w, q} \leq \Omega_\varphi^k(f,\delta)_{w, s}^\theta \Omega_\varphi^k(f,\delta)_{w, p}^{1-\theta} ,
\]
with similar inequalities holding for $\overrightarrow \Omega_\varphi^k$ and $\overleftarrow \Omega_\varphi^k$ as well. Therefore,
\begin{eqnarray*}
  \omega_\varphi^k(f,\delta)_{w,q}  &=&     \Omega_\varphi^k(f,\delta)_{w,q}
 +  \overrightarrow \Omega_\varphi^k(f,\delta)_{w,q} + \overleftarrow\Omega_\varphi^k(f,\delta)_{w,q} \\
 & \leq & \Omega_\varphi^k(f,\delta)_{w, s}^\theta \Omega_\varphi^k(f,\delta)_{w, p}^{1-\theta} +
 \overrightarrow \Omega_\varphi^k(f,\delta)_{w, s}^\theta \overrightarrow \Omega_\varphi^k(f,\delta)_{w, p}^{1-\theta} +
 \overleftarrow\Omega_\varphi^k(f,\delta)_{w, s}^\theta \overleftarrow\Omega_\varphi^k(f,\delta)_{w, p}^{1-\theta} \\
 & \leq & 3\,  \omega_\varphi^k(f,\delta)_{w, s}^\theta \, \omega_\varphi^k(f,\delta)_{w, p}^{1-\theta} .
 \end{eqnarray*}
 Hence, using \ineq{direct} and \thm{maintheorem}
we have the following estimates for $f\in\M^k\cap \W_{p, r}^{\alpha+r/2,\beta+r/2}$, $0\leq r \leq k-1$:
\begin{eqnarray*}
E_n(f)_{\wab,q}   &\leq&   c n^{-r} \, \omega_\varphi^{k-r}(f^{(r)},n^{-1})_{w_{\alpha+r/2,\beta+r/2}, q} \\ \nonumber
&\leq&
c n^{-r} \, \omega_\varphi^{k-r}(f^{(r)},n^{-1})_{w_{\alpha+r/2,\beta+r/2}, s}^\theta \, \omega_\varphi^{k-r}(f^{(r)},n^{-1})_{w_{\alpha+r/2,\beta+r/2}, p}^{1-\theta} \\ \nonumber
&\leq&
c n^{-r} \, \left[\Upsilon_{1/n}^{\alpha+r/2,\beta+r/2}(k-r,s,p)\right]^{\theta}  \, \norm{w_{\alpha+r/2,\beta+r/2}f^{(r)}}{p}^\theta  \, \omega_\varphi^{k-r}(f^{(r)},n^{-1})_{w_{\alpha+r/2,\beta+r/2}, p}^{1-\theta} .
\end{eqnarray*}
If   $s$  is such that $1<s<q$ and $s\ne  p/2$, then
\[
\Upsilon_{1/n}^{\alpha+r/2,\beta+r/2}(k-r,s,p) =
\begin{cases}
n^{-2/s+2/p}\,, &   \mbox{\rm if  $k-r\geq 2$,}\\
n^{-2/s+2/p}\,, &   \mbox{\rm if   $k-r=1$ and $p < 2s$, }\\
n^{-1/s} \,, &  \mbox{\rm if $k-r=1$ and $p>2s$,}
\end{cases}
\]
and so
\[
\left[\Upsilon_{1/n}^{\alpha+r/2,\beta+r/2}(k-r,s,p)\right]^\theta =
\left\{
\begin{array}{ll}
n^{-2/q + 2/p}\,, & \quad \mbox{\rm if $0\leq r\leq k-2$, or $r=k-1$ and $p<2s$,} \\
n^{-(p-q)/q(p-s)} \,, & \quad \mbox{\rm if $r=k-1$ and $p > 2s$.}
\end{array} \right.
\]
We now note that one can choose $s$ so that $1<s<q$ and $p<2s$ iff $p<2q$. Also, note that, for any $s>1$,
$\left[\Upsilon_{1/n}^{\alpha+(k-1)/2,\beta+(k-1)/2}(1,s,\infty)\right]^\theta = n^{-1/q}$.

Therefore, taking into account that, in
the case $p<\infty$, $\omega_\varphi^{k-r}(f^{(r)},n^{-1})_{w_{\alpha+r/2,\beta+r/2}, p} \to 0$ as $n\to \infty$, and that
$\omega_\varphi^{k-r}(f^{(r)},n^{-1})_{w_{\alpha+r/2,\beta+r/2}, \infty} \to 0$ as $n\to \infty$ provided that  $f^{(r)}$ is continuous on $(-1,1)$ and $\lim_{x \pm 1}  w_{\alpha+r/2,\beta+r/2}(x) f^{(r)}(x) = 0$,
 we have
the following two corollaries of \thm{maintheorem}.

\begin{corollary} \label{cor8.6}
Let $k\in\N$, $1<q<p<\infty$, $0\leq r \leq k-1$, $\alpha,\beta\geq 0$, and let  $f\in\M^k \cap  \W_{p, r}^{\alpha+r/2,\beta+r/2}$. Then
\[
E_n(f)_{\wab,q}  = o\left( n^{-r-2/q+2/p}   \right) ,\quad n \to \infty ,
\]
where either $0\leq r\leq k-2$, or $r=k-1$ and $p<2q$.
\end{corollary}

\begin{corollary} \label{cor8.7}
Let $k\in\N$, $1<q<\infty$, $0\leq r \leq k-1$, $\alpha,\beta\geq 0$, and let $f\in\M^k$ be such that  $f^{(r)}$ is continuous on $(-1,1)$ and $\lim_{x \pm 1}  w_{\alpha+r/2,\beta+r/2}(x) f^{(r)}(x) = 0$. Then
\[
E_n(f)_{\wab,q}  = o\left( n^{-r-\min\{k-r,2\}/q}   \right) , \quad n \to \infty .
\]
\end{corollary}

\sect{Lower estimates of polynomial approximation}

 The following Remez-type inequality follows from \cite[(7.16), (6.10)]{mt}.

\begin{theorem}
Let $1\leq p \leq  \infty$, and let $w$ be a doubling weight in the case $1\leq p<\infty$ or an $A^*$ weight in the case $p=\infty$. For every $\Lambda\leq n$,
there is a constant $C=C(\Lambda)$  such that, if $E \subset [-1,1]$ is an interval and $ \int_E (1-x^2)^{-1/2} dx   \leq \Lambda/n$, then, for each $p_n\in\Pn$, we have
\[
\int_{-1}^1 |p_n(x)|^p w(x)\, dx \leq C \int_{[-1,1]\setminus E} |p_n(x)|^p w(x) \, dx , \quad \mbox{\rm if } \; 1\leq p <\infty,
\]
or
\[
\norm{p_n w}{\LL_\infty[-1,1]} \leq C \norm{p_n w}{\LL_\infty([-1,1]\setminus E)}, \quad \mbox{\rm if } \; p= \infty.
\]
\end{theorem}
 We recall that $w$ is a doubling weight if
$
  \int_{2I\cap [-1,1]} w(x)\, dx  \leq  L \int_I w(x)\, dx  ,
$
for all intervals $I\subset [-1,1]$ ($2I$ is the interval twice the length of $I$ and with midpoint at the midpoint of $I$), and it is an $A^*$ weight if, for all intervals $I\subset [-1,1]$ and $x\in I$,
$w(x) \leq L  \int_I w(x)\, dx/|I| $.

Since $\wab^p$, $\alpha,\beta > -1/p$, is a doubling weight, and $\wab$, $\alpha,\beta \geq 0$, is an $A^*$ weight, we immediately get the following corollary  (see also \cite{emn}).

\begin{corollary} \label{remez}
Let $1\leq p \leq  \infty$, and let $\alpha,\beta \in J_p$. For every $\Lambda\leq n$,
there is a constant $C=C(\Lambda)$  such that, if $E \subset [-1,1]$ is an interval and $ \int_E (1-x^2)^{-1/2} dx   \leq \Lambda/n$, then, for each $p_n\in\Pn$, we have
\[
\norm{p_n \wab}{\Lp[-1,1]} \leq C \norm{p_n \wab}{\Lp([-1,1]\setminus E)}.
\]
\end{corollary}

We are now ready to construct (truncated power) functions which will yield lower estimates.
Note that, if $k\in\N$, $1\leq p\leq \infty$, $0\leq r\leq k-1$, $0\leq \xi <1$, $\alpha\in\R$, $\beta\in J_p$ and  $f(x):= (x-\xi)_+^{k-1}$, then
\be \label{trtr}
\norm{\wab f^{(r)}}{p} \sim (1-\xi)^{\beta +k-r-1 +1/p} .
\ee

\begin{lemma} \label{eneg}
Let $1\leq  q \leq \infty$, $k\in\N$,   $\alpha,\beta\geq 0$,   $n\geq 2k$, $0\leq \xi \leq 1-2k^2n^{-2}$ and let $f(x):= (x-\xi)_+^{k-1}$. Then
\[
E_n(f)_{\wab,q} \geq    c n^{-k+1-1/q} (1-\xi)^{\beta   + (k-1) /2 +1/(2q)},
\]
for some constant $c$ independent of $n$.
\end{lemma}

\begin{proof} We only provide the proof for the case $q<\infty$. If $q=\infty$, it is obvious what modifications are needed.
It is convenient to denote $\theta_n := k \varphi(\xi)/(2n)$. Then, in particular, $\theta_n \leq 1/4$ and $\xi \pm 2\theta_n \in [-1,1]$.
Now, let $p_n$ be an arbitrary polynomial from $\Pn$,   define
\begin{eqnarray*}
\tilde f_n(x) := \Delta_{\varphi(\xi)/n}^k (f, x) , &&  \quad f_n(x):= \tilde f_n((1-\theta_n)x),\\
\tilde q_n(x) := \Delta_{\varphi(\xi)/n}^k (p_n, x), && \quad q_n(x):= \tilde q_n((1-\theta_n)x) ,
\end{eqnarray*}
and note that $\tilde q_n$
is a polynomial of degree $\leq n$   on $J_n:= \left[-1+\theta_n, 1-\theta_n \right]$, and hence
$q_n$ is a polynomial of degree $\leq n$   on $[-1,1]$. We also note that
$\tilde f_n(x) = 0$, for $x\not\in \tilde I_n:= \left[\xi - \theta_n, \xi + \theta_n   \right] \subset J_n$, and hence
$  f_n(x) = 0$, for $x\not\in I_n:= \left[(\xi - \theta_n)/(1-\theta_n), (\xi + \theta_n)/(1-\theta_n)   \right] \subset [-1,1]$.

Now,
\begin{eqnarray*}
\norm{\wab (f_n -q_n )}{q}^q &=& \int_{-1}^1 \wab^{q}(x) | \tilde f_n((1-\theta_n)x) - \tilde q_n((1-\theta_n)x)|^q \, dx \\
& \leq & c
\int_{-1+\theta_n}^{1-\theta_n} \wab^q(x/(1-\theta_n)) | \tilde f_n(x) - \tilde q_n(x)|^q \, dx \\
& \leq &
c \int_{-1+\theta_n}^{1-\theta_n}  \wab^q(x/(1-\theta_n))   \sum_{i=0}^k \left| f (x-\theta_n +  i\varphi(\xi)/n) - p_n(x-\theta_n +  i\varphi(\xi)/n) \right|^q \, dx \\
& \leq & c \sum_{i=0}^k \int_{-1+ i\varphi(\xi)/n}^{1-2\theta_n+i\varphi(\xi)/n}     \wab^q\left(  (y + \theta_n -  i\varphi(\xi)/n)/(1-\theta_n)  \right)  |f(y)-p_n(y)|^q \, dy \\
& \leq & c \norm{\wab (f-p_n)}{q}^q ,
\end{eqnarray*}
since $\wab \left(  (y + \theta_n -  i\varphi(\xi)/n)/(1-\theta_n) \right) \leq c \wab(y)$.

 It is straightforward to check that  $\int_{I_n} (1-x^2)^{-1/2} \, dx \leq c(k)/n$, and so  \cor{remez}
implies that
\[
\norm{\wab q_n}{q} \leq c \norm{\wab q_n}{\Lq([-1,1]\setminus I_n)} .
\]
Therefore, recalling that $f_n(x)=0$, $x \in [-1,1]\setminus I_n$, we have
\begin{eqnarray*}
\norm{\wab  f_n  }{q} & \leq & \norm{\wab  (f_n-q_n)  }{q} + \norm{\wab  q_n  }{q} \\
&\leq& \norm{\wab  (f_n-q_n)  }{q} + c \norm{\wab (f_n-q_n)}{\Lq([-1,1]\setminus I_n)}\\
& \leq & c \norm{\wab (f_n-q_n)  }{q} \\
& \leq & c \norm{\wab (f-p_n)}{q} .
\end{eqnarray*}
Now, noting that $\tilde f_n(x)=   f(x+\theta_n) = (x+\theta_n-\xi)^{k-1}$, if
$x\in [\xi-\theta_n, \xi - \theta_n +  \varphi(\xi)/n]$
 we have
\begin{eqnarray*}
\norm{\wab  f_n  }{q}^q
& \geq &
c \int_{-1+\theta_n}^{1-\theta_n} \wab^q(x/(1-\theta_n))   |\tilde f_n (x)|^q \, dx \\
 & \geq  &
  c \int_{\xi-\theta_n}^{\xi - \theta_n+ \varphi(\xi)/n} (1-\theta_n-x)^{\beta q} (x+\theta_n-\xi)^{(k-1)q} \, dx\\
 & \geq  &
  c \int_{\xi}^{\xi+ \varphi(\xi)/n} (1-y)^{\beta q} (y-\xi)^{(k-1)q} \, dy\\
& \geq & c n^{-(k-1)q-1} (1-\xi)^{\beta q + (k-1)q/2 +1/2}  ,
\end{eqnarray*}
and so $\norm{\wab  f_n  }{q} \geq c n^{-k+1-1/q} (1-\xi)^{\beta   + (k-1) /2 +1/(2q)}  $.

 Hence, for any $p_n\in\Pn$,
\[
\norm{\wab (f-p_n)}{q} \geq c n^{-k+1-1/q} (1-\xi)^{\beta   + (k-1) /2 +1/(2q)} ,
\]
and the proof is complete.
\end{proof}

 The following two corollaries provide all lower estimates in \thm{mpoly} and show that none of the powers of $n$ in \cor{corrjack} can be decreased.

 \begin{corollary} \label{corrr1}
Let $1\leq p, q \leq \infty$, $k\in\N$,  and  $\alpha,\beta\geq 0$.
Then, there exists a function $f\in\M^k\cap \W_{p, k-1}^{\alpha+(k-1)/2,\beta+(k-1)/2}$ such that, for each $n\in\N$,
\be \label{trine}
E_n(f)_{\wab,q} \geq  c  n^{-k-1/q+1} \norm{w_{\alpha+(k-1)/2, \beta+(k-1)/2} f^{(k-1)}}{p} ,
\ee
for some constant $c$ independent of $n$.
\end{corollary}

\begin{proof} We let $f(x) := x_+^{k-1}$ and note that $f\in\M^k$. Now, \ineq{trtr} implies that
$\norm{w_{\alpha+(k-1)/2, \beta+(k-1)/2} f^{(k-1)}}{p} \sim 1$, and
 \lem{eneg} implies
$E_n(f)_{\wab,q} \geq c n^{-k-1/q+1}$, for $n\geq 2k$. For $1\leq n < 2k$,  \ineq{trine} follows from
$E_n (f)_{\wab,q} \geq E_{2k} (f)_{\wab,q} \geq c$.
\end{proof}

It follows from \cor{cor8.7} that there does not exist $f\in\C^{k-1}(-1,1)\cap\M^k$ which is independent of $n$, satisfies
$\lim_{x \pm 1}  w_{\alpha+r/2,\beta+r/2}(x) f^{(r)}(x) = 0$, and for which \ineq{trine} holds.

\begin{corollary} \label{corrr2}
Let $1\leq p, q \leq \infty$, $k\in\N$, $0\leq r \leq k-1$, $\alpha,\beta \geq 0$,  and $n\in\N$. Then, there exists a function $f_n\in\M^k\cap \W_{p, r}^{\alpha+r/2,\beta+r/2}$ such that
\[
E_n(f_n)_{\wab,q} \geq  c  n^{-r-2/q+2/p} \norm{w_{\alpha+r/2, \beta+r/2} f_n^{(r)}}{p} ,
\]
for some constant $c$ independent of $n$.
\end{corollary}

\begin{proof} For $1\leq n < 2k$, the statement is clearly true, for example, for $f_n(x)=x_+^{k-1}$. If $n\geq 2k$, we let $\xi_n = 1- 2k^2 n^{-2}$ and $f_n(x) :=   (x-\xi_n)_+^{k-1}$. Then $f_n  \in \M^k$,
\lem{eneg} implies that
$
E_n(f_n)_{\wab,q} \geq c n^{-2\beta - 2k + 2 - 2/q} ,
$
and \ineq{trtr} yields
$
\norm{w_{\alpha+r/2, \beta+r/2} f_n^{(r)}}{p}  \sim   n^{-2\beta-2k-2/p+2+r} .
$
Therefore,
$ E_n(f_n)_{\wab,q}/\norm{w_{\alpha+r/2, \beta+r/2} f_n^{(r)}}{p}
\geq c  n^{-r-2/q+2/p}$.
\end{proof}

It is interesting to note that \cor{cor8.6} implies that $f_n$ in \cor{corrr2} cannot be replaced by a function which is independent of $n$.

\end{document}